\documentclass[11pt, leqno]{amsart}

\usepackage[T1]{fontenc}

\usepackage{amsfonts, delarray, amssymb, amsmath, amsthm, a4, a4wide}
\usepackage{thmtools, cite}

\usepackage{mathrsfs}
\usepackage{comment}
\usepackage{mathtools}
\usepackage{graphicx}
\usepackage{latexsym}
\usepackage{epsfig}
\usepackage{color}

\usepackage[plainpages=false,hypertexnames=false,pdfpagelabels]{hyperref}
\definecolor{medium-blue}{rgb}{0,0,.8}

\hypersetup{colorlinks, linkcolor={blue}, citecolor={medium-blue}, urlcolor={medium-blue}}

\hypersetup{hidelinks}

\numberwithin{equation}{section}
\DeclareMathOperator{\supp}{supp}

\DeclareMathOperator{\diam}{\text{diam}}

\DeclareMathOperator{\dist}{\text{dist}}
\newcommand{\e}{\varepsilon}
\newcommand{\p}{\partial}

\def\R{\mathbb{R}}

\newtheorem{thm}{Theorem}[section]
\newtheorem{cor}[thm]{Corollary}
\newtheorem{lem}[thm]{Lemma}
\newtheorem{prop}[thm]{Proposition}

\theoremstyle{definition}
\newtheorem{defn}[thm]{Definition}

\newtheorem{rem}[thm]{Remark}
\newtheorem{rems}[thm]{Remarks}

\begin{document} 

\title[Degenerate Monge--Amp\`ere equations with singular measures]
{Uniqueness and sharp boundary estimates for degenerate Monge--Amp\`ere equations with singular measures}
\author{Chong Gu}
\address{Department of Mathematics, Indiana University, Bloomington, IN 47405, USA}
\email{chongu@iu.edu}
\thanks{The author was supported in part by the National Science Foundation under grant DMS-2452320.}

\begin{abstract}
We study the uniqueness and boundary behavior of nonzero convex Aleksandrov solutions to $\det D^2 u=M|u|^p\nu$ with zero boundary values on bounded convex domains in $\mathbb{R}^n (n \geq 2)$. 
For $0<p<n$, we prove the uniqueness of nonzero convex solutions in the finite-energy class when $\nu$ is a locally finite Borel measure with positive mass and $\int_\Omega \text{dist }(\cdot,\partial\Omega)\,d\nu<\infty$. 
For $p>n$, we construct an explicit two-shell measure on the unit ball for which the problem has at least three radial solutions that are globally Lipschitz and have finite energy. 
In the case of $\nu=\text{dist }(\cdot,\partial\Omega)^{-\alpha}\,d\mathcal{L}^n$, $0\leq\alpha<2$, we prove global Lipschitz continuity when $p-\alpha>n-2$ and obtain sharp upper and lower estimates on domains with a flat boundary part when $n(\alpha-1)-2<p-\alpha\leq n-2$. 
When $\alpha=0$ and $p=n-2$, our log-Lipschitz lower estimate has the same exponent as the known upper estimate. This answers the question raised by Le (Global Lipschitz and Sobolev estimates for the Monge--Amp\`ere eigenfunctions of general bounded convex domains.
{\it Ann. Fac. Sci. Toulouse Math.} (6) {\bf 35} (2026)). 
We also give a sufficient condition for finite Monge--Amp\`ere energy on every bounded convex domain, prove its necessity when the boundary contains a flat part, and apply it to prove the uniqueness of the Monge--Amp\`ere eigenvalue among all nonzero convex solutions.
\end{abstract}

\maketitle

\section{Introduction and statement of the main results}
\label{sec:introduction}
We study the uniqueness of nonzero convex solutions and their boundary estimates for degenerate Monge--Amp\`ere equations with singular Borel measures on bounded convex domains in Euclidean space.

Let $\nu$ be a locally finite Borel measure with positive total mass on a bounded convex domain $\Omega\subset\R^n$ $(n\geq2)$, and let $p, M > 0$. We consider nonzero convex solutions to the Monge--Amp\`{e}re equation
\begin{equation}\label{eq:intro-general}
\det D^2u  =  M|u|^p \nu \quad \text{in } \Omega, \qquad u = 0 \quad \text{on } \partial \Omega.
\end{equation}
Here $u\in C(\overline{\Omega})$ is an unknown convex function, and the equation is understood in the Aleksandrov sense; that is, 
\[
\mu_u = M|u|^p \nu,
\]
where $\mu_u$ is the Monge--Amp\`ere measure of $u$; see \eqref{eq:MA_measure}.
By homogeneity, when $p\neq n$ we can normalize $M=1$ by rescaling $u$. When $p=n$, the equation is scale invariant and $M$ plays the role of the Monge--Amp\`ere eigenvalue.

\subsection{History and related work}\label{sec:history} We recall some background on our equation \eqref{eq:intro-general}.
This equation arises in several contexts, such as the Monge--Amp\`ere eigenvalue problem \cite{Lions,Tso,Le_18,TongYau}, the analysis of Abreu's equation in the constant scalar curvature problem \cite{Donaldson},
the $L_p$-Minkowski problem and centroaffine geometry \cite{ChouWang}, and the study of convex envelopes \cite{CTW}.

\medskip
When $\nu$ is the $n$-dimensional Lebesgue measure $d\mathcal{L}^n$, equation~\eqref{eq:intro-general} becomes
\begin{equation}\label{eq: leb_eq}
\det D^2u=M|u|^p\quad\text{in }\Omega,\qquad u=0\quad\text{on }\partial\Omega,
\end{equation}
which has been extensively studied. When $p=n$, this is the Monge--Amp\`ere eigenvalue problem, and $M$ is called the Monge--Amp\`ere eigenvalue of $\Omega$. 
On smooth uniformly convex domains, Lions \cite{Lions} proved the existence and uniqueness of the eigenvalue, with an eigenfunction $u \in C^\infty(\Omega)\cap C^{1, 1}(\overline{\Omega})$ unique up to positive multiplication. 
Hong, Huang, and Wang \cite{HHW} proved that these eigenfunctions are smooth up to the boundary in dimension two.
Le and Savin \cite[Theorem 1.4]{LeSavin} established $u\in C^\infty(\overline{\Omega})$ on smooth bounded uniformly convex domains in all dimensions.
Tso \cite{Tso} obtained a variational characterization of the eigenvalue. He also proved the existence of a nonzero convex solution to \eqref{eq: leb_eq} for $0<p\ne n$ and the uniqueness of such solutions for $0<p<n$. 
Later, Le extended the existence result with $u \in C^\infty(\Omega) \cap C(\overline{\Omega})$ for $0<p\ne n$ and the uniqueness result for $0<p<n$ to arbitrary bounded convex domains; see \cite[Theorem 4.2 and Proposition 4.5]{Le_18}. For $p=n$, he proved the uniqueness of the eigenvalue and of its eigenfunctions up to positive multiplication. 

On smooth bounded strictly convex domains containing the origin,
Tong and Yau \cite{TongYau} extended the variational approach
to generalized Monge--Amp\`ere functionals, obtaining a unique
eigenvalue and eigenfunctions unique up to positive scaling
for an associated nonlinear eigenvalue problem.
See also the related study by Collins and Firester \cite{CollinsFirester}.

\smallskip
For $p>n$, Huang \cite{Huang} proved the uniqueness of least-energy solutions on smooth uniformly convex domains when $n=2$. For $n\geq2$, he also proved the uniqueness of all nonzero convex solutions on smooth uniformly convex domains when $n<p<\alpha_0(n)$ for some $\alpha_0(n)>n$. When $n=2$, Cheng, Huang, and Xu \cite{CHX} proved the uniqueness of nonzero convex solutions on bounded convex domains having at least two distinct axes of symmetry. Zhou \cite[Theorem 1.1]{Zhou_26} recently proved that, for every $p>n$, equation~\eqref{eq: leb_eq} has at most one nonzero convex solution on any bounded convex domain. Together with Le's existence result, this gives the existence and uniqueness of a nonzero convex solution for all $p>n$.

\smallskip
It is known that the solutions to \eqref{eq: leb_eq} are smooth in the interior; that is, $u \in C^\infty(\Omega)$.  Le also obtained boundary H\"older estimates for $0<p<n-2$, Lipschitz estimates for $p>n-2$, and log-Lipschitz estimates for $p=n-2$; see \cite{Le_22,Le_23,Le_EVP}. We recall these results below when discussing the distance-weighted measure.

\smallskip
For general Borel measures $\nu$, Lu and Zeriahi \cite{LZ} studied the cases $p=0$ and $p=n$ using finite-energy methods from complex Monge--Amp\`ere theory. 
On bounded, smooth, uniformly convex domains containing the origin, He and Huang \cite[Theorem 1.1]{HeHuang} obtained related existence and uniqueness results for the Monge--Amp\`ere eigenvalue problem with $\nu=|x|^s\,d\mathcal{L}^n$, $s>-n$.

Le \cite{Le_Var} subsequently developed a direct real variational theory based on mixed Monge--Amp\`ere measures and monotonicity. In particular, for $0\leq p<n$, \cite[Theorem 1.1(ii)--(iii)]{Le_Var} gives a nonzero continuous Aleksandrov solution to \eqref{eq:intro-general} when
\[
\nu(\Omega)>0\qquad\text{and}\qquad \int_\Omega \dist(\cdot,\partial\Omega)\,d\nu<\infty.
\]
For $p>n$, \cite[Theorem 1.1(i)]{Le_Var} gives a nonzero solution in the finite-energy class under the condition
\[
\int_\Omega \dist(\cdot,\partial\Omega)^{\frac{p+1}{n+1}}\,d\nu<\infty.
\]
When $p=n$, Le proved comparison of eigenvalues in the finite-energy class and the uniqueness of eigenfunctions up to positive multiplication; see \cite[Lemma 6.5 and Theorem 6.14]{Le_Var}. His examples outside this class show that the energy assumption cannot simply be omitted.

As recalled above, in the Lebesgue case, the uniqueness of nonzero convex solutions is known for $0<p<n$ by \cite[Proposition 4.5]{Le_18} and for $p>n$ by \cite[Theorem 1.1]{Zhou_26}.
When $p=n$, the eigenvalue is unique and the eigenfunctions are unique up to positive multiplication; see \cite[Theorem 1.1]{Le_18}.
{\it For general Borel measures, Le's existence result for $0<p<n$ leaves open the question of the uniqueness of nonzero convex solutions in the finite-energy class.
For $p>n$, we ask whether finite energy can still ensure the uniqueness of nonzero convex solutions.}
We will address these issues in Theorems \ref{thm: p<n_uniq} and \ref{thm:supercritical-two-shell}.

\medskip
For a deeper understanding of the boundary behavior of solutions to \eqref{eq:intro-general},
we restrict ourselves to a special class of singular Borel measures $\nu$ of the form $\nu=\dist(\cdot,\partial\Omega)^{-\alpha}\,d\mathcal{L}^n$.
Equation \eqref{eq:intro-general} with this measure appears in Savin--Zhang \cite{SavinZhang} and Le--Savin \cite{LeSavinSingular}.
See also related results in Mohammed \cite{Mohammed}, Cheng--Yau \cite{ChengYau} and Gilbarg--Trudinger \cite[Section 17.7]{GT}. 

We consider the problem
\begin{equation}\label{main_eq}
\left\{\begin{aligned}
\operatorname{det} D^2 u & = M |u|^p \dist(\cdot,\partial\Omega)^{-\alpha} & & \text { in } \Omega, \\
u & =0 & & \text { on }\partial \Omega.
\end{aligned}\right.
\end{equation}
Here $M>0$, $p>0$, and $0\leq\alpha<2$. 

For $\alpha>0$, the weight is singular at the boundary, while $\alpha<2$ is exactly the condition that
\[
\int_\Omega \dist(\cdot,\partial\Omega)\,d\nu=\int_\Omega \dist(\cdot,\partial\Omega)^{1-\alpha}\,dx<\infty.
\]
For $0<p<n$, this gives the solvability of \eqref{main_eq} by \cite[Theorem 1.1(iii)]{Le_Var}. When $p>n$, the hypothesis in \cite[Theorem 1.1(i)]{Le_Var} also holds because
\[
\int_\Omega \dist(\cdot,\partial\Omega)^{\frac{p+1}{n+1}}\,d\nu
=\int_\Omega \dist(\cdot,\partial\Omega)^{\frac{p+1}{n+1}-\alpha}\,dx<\infty.
\]
When $p=n$, Theorem 1.1(ii), Lemma 6.2, and Theorem 1.1(iv) of \cite{Le_Var} give an eigenpair $(M,u)$. Thus a nonzero solution exists for every fixed $M>0$ when $p\ne n$, and for a suitable $M>0$ when $p=n$, even though the measure may have infinite total mass.

\medskip
When $\alpha=0$, equation~\eqref{main_eq} reduces to \eqref{eq: leb_eq}. We recall the boundary estimates that motivate our results. We assume $M=1$ when $p\ne n$.
\begin{itemize}
\item When $p<n-2$, Le \cite[Theorem 1.1]{Le_23} proved that the gradient of the unique nonzero convex solution $u$ blows up near every flat part of the boundary. 
On the other hand, for $0<p<n-2$ and $n\geq3$, Le \cite[Proposition 2.2]{Le_EVP} proved that

\[
|u(x)|\leq C(n,p,\Omega)\dist(x,\partial\Omega)^{\frac{2}{n-p}},
\qquad x\in\Omega.
\]

\item When $p>n-2$, Le \cite[Theorem 1.1]{Le_EVP} proved that convex solutions are globally Lipschitz, with
\[
|u(x)|\leq C(n,\Omega,p,M)\dist(x,\partial\Omega)\|u\|_{L^\infty(\Omega)},
\qquad x\in\Omega.
\]

\item When $p=n-2$, the explicit solution on a planar triangle in $\R^2$ is only log-Lipschitz; see \cite[Example 3.32]{Le_24}. For $n\geq3$, Le \cite[Theorem 1.5]{Le_EVP} proved
\begin{equation}\label{Le_n-2_up}
|u(x)|\leq C(n,\Omega)\dist(x,\partial\Omega)\left(1+|\log \dist(x,\partial\Omega)|^{n/2}\right), \qquad x\in\Omega,
\end{equation}
and, near a flat part $\Gamma$ of $\partial\Omega$,
\begin{equation}\label{Le_n-2_low}
|u(x)|\geq c(n,\Omega,\Gamma)\dist(x,\partial\Omega)|\log \dist(x,\partial\Omega)|^{1/n}.
\end{equation}
\end{itemize}
{\it The known upper and lower bounds therefore left a gap between the exponents $n/2$ and $1/n$ of $|\log \dist(x,\partial\Omega)|$ when $p = n-2$.}

For $\alpha>0$, the factor $|u|^p$ vanishes at the boundary while $\dist(\cdot,\partial\Omega)^{-\alpha}$ diverges. We determine how this changes the Lipschitz threshold and upper and lower boundary estimates.
This will be addressed in Theorems \ref{thm:intro-Lip-alpha} and \ref{thm:intro-inf-bd}.

\subsection{Standard notions}

Before stating our main results, we recall some standard notions. 
For a convex function $u$ on a convex domain $\Omega \subset \R^n$, its subdifferential at $x\in\Omega$ is
\[
\partial u(x):=\left\{q\in\R^n:u(y)\geq u(x)+q\cdot(y-x)\ \text{for all }y\in\Omega\right\}.
\]
For a Borel set $E\subset\Omega$, the Monge--Amp\`ere measure of $u$ is
\begin{equation}\label{eq:MA_measure}
\mu_u(E):=|\partial u(E)|,\qquad \partial u(E):=\bigcup_{x\in E}\partial u(x).
\end{equation}
The finite-energy class is
\[
\mathbb E(\Omega):=\left\{w\in C(\overline{\Omega}):w\text{ is convex in }\Omega,\ w=0\text{ on }\partial\Omega,\ \int_\Omega |w|\,d\mu_w<\infty\right\}.
\]
We have the inclusion
\begin{equation}\label{eq:C01_in_E}
\{w \in C^{0, 1}(\overline{\Omega}): w \text{ is convex in } \Omega, w = 0 \text{ on } \partial \Omega\} \subset \mathbb{E}(\Omega).
\end{equation}
For $w\in\mathbb E(\Omega)$, we denote its Monge--Amp\`ere energy by
\[
I[w]=I[w;\Omega]:=\int_\Omega |w|\,d\mu_w.
\]

We use $c=c(\ast,\ldots,\star)$ and $C= C(\ast,\ldots,\star)$ to denote positive constants $c, C$ depending on the quantities appearing in the parentheses; they may change from line to line. 

\subsection{Main results}

We first state the results for general Borel measures. 
The following theorem gives the uniqueness of nonzero convex solutions in the finite-energy class for $0<p<n$.

\begin{thm}[Uniqueness for subcritical Monge--Amp\`ere equations in the finite-energy class]\label{thm: p<n_uniq}
Let $\Omega \subset \R^n (n \geq 2)$ be a bounded convex domain.
Let $p \in (0, n)$ and $\nu$ be a locally finite Borel measure on $\Omega$ with
\[
\nu (\Omega) > 0 \quad \text{and} \quad \int_\Omega \dist(\cdot, \partial \Omega) \; d\nu < \infty.
\]
Suppose $u, v \in \mathbb E(\Omega)\setminus \{0\}$ satisfy
\[
\mu_u = |u|^p \nu \quad \text{and} \quad \mu_v = |v|^p \nu \quad \text{in } \Omega, \quad \text{and } u = v = 0 \quad \text{on } \partial \Omega.
\]
Then $u = v$ on $\Omega$.
\end{thm}

\begin{rem}
For $\nu=d\mathcal{L}^n$, Theorem~\ref{thm: p<n_uniq} recovers the uniqueness result recalled in Section~\ref{sec:history}. 
When $\nu$ is compactly supported in $\Omega$ or has a strictly positive continuous density on $\overline\Omega$, 
Le's comparison principle \cite[Theorem 1.5]{Le_Var} (see Theorem~\ref{thm:subcritical-comparison}) implies the uniqueness of nonzero convex solutions to \eqref{eq:intro-general}; obviously, these solutions have finite energy.
Theorem~\ref{thm: p<n_uniq} extends this uniqueness conclusion to locally finite Borel measures satisfying the distance integrability condition, provided the solutions have finite energy.
Such measures may be singular with respect to the Lebesgue measure and may have infinite total mass. 
Furthermore, one might wonder if the finite-energy assumptions, that is $u, v \in \mathbb{E}(\Omega)$, can be removed in Theorem~\ref{thm: p<n_uniq}. 

\end{rem}

We prove Theorem~\ref{thm: p<n_uniq} in Section~\ref{sec:subcritical-uniqueness}. 

\medskip
For $p>n$, we show that the uniqueness of nonzero convex solutions can fail for general Borel measures even among Lipschitz solutions with finite energy. We construct an example in the unit ball by combining normalized surface measures on two concentric spheres. Let $B_r:=B_r(0)\subset\R^n$ be the ball centered at the origin with radius $r>0$, and let 
\[
\omega_n:=|B_1|.
\]
For $0<r<1$, let
\begin{equation}\label{eq:normalized-sphere-measure}
 \sigma_r
 :=\frac{\mathcal{H}^{n-1}\lfloor \partial B_r}{n\omega_n r^{n-1}}
\end{equation}
be the normalized surface measure on $\partial B_r$.

\begin{thm}[Nonuniqueness for supercritical Monge--Amp\`ere equations with compactly supported measures]
\label{thm:supercritical-two-shell}
Let $p > n\geq2$. Then there exist radii $0<r_1<r_2<1$
and constants $A_1,A_2>0$, depending only on $n$ and $p$, such that, for the finite positive Borel measure
\[
 \nu=A_1\sigma_{r_1}+A_2\sigma_{r_2},
\]
the Dirichlet problem
\[
 \mu_u=|u|^p\nu\quad\text{in }B_1,
 \qquad
 u=0\quad\text{on }\partial B_1
\]
has at least three nonzero distinct radial convex solutions
\[
 u_1,u_2,u_3\in C^{0,1}(\overline{B_1}).
\]
\end{thm}

\begin{rem}
The measure $\nu$ in this theorem is compactly supported in $B_1$, so
\[
 \int_{B_1}\dist(\cdot,\partial B_1)^\gamma\,d\nu<\infty
 \qquad\text{for every }\gamma\in\R.
\]
Moreover, the solutions have finite energy; see \eqref{eq:C01_in_E}. 
Thus this two-shell example gives nonuniqueness of solutions with finite energy even though the measure has no boundary singularity.
This is quite different from the Lebesgue measure case studied in Zhou \cite{Zhou_26}.
The one-dimensional examples in \cite{Le_Var} deal with $p=n$ and use a measure singular at the boundary; their eigenfunctions have infinite energy.

\end{rem}

We prove Theorem~\ref{thm:supercritical-two-shell} in Section~\ref{sec:supercritical-nonuniqueness}. 

\medskip
We next turn to the distance-weighted problem. The following theorem gives global upper estimates on general bounded convex domains. 

\begin{thm}[Global H\"older estimates]\label{thm:intro-Lip-alpha}\label{thm:intro-refined-holder}\label{thm: refined_critical}
Let $\Omega \subset \R^n (n \geq 2)$ be a bounded convex domain, $0 \leq \alpha < 2$, $p>0$, and $M>0$. Let $u \in C(\overline{\Omega})$ be a nonzero convex solution to
\[
\mu_u=M|u|^p \dist(\cdot,\partial\Omega)^{-\alpha}\,d\mathcal{L}^n\quad\text{in }\Omega,
\qquad u=0\quad\text{on }\partial\Omega.
\]
\begin{itemize}
        \item[(i)] If $p-\alpha> n - 2$, then $u$ is globally Lipschitz with the estimate
                \[
                |u(x)| \leq C(n, \Omega, \alpha, p) \dist(x,\partial\Omega) \|u\|_{L^\infty(\Omega)}.
                \]

        \item[(ii)] If $p-\alpha = n-2$, then $u$ is globally log-Lipschitz with correction of order $\frac{n}{2-\alpha}$ and the estimate
                \[
                |u(x)| \leq C(n,\Omega, \alpha, p, M) \dist(x,\partial\Omega)\left(1 + |\log \dist(x,\partial\Omega)|^{\frac{n}{2-\alpha}}\right).
                \]
        \item[(iii)] If $p-\alpha < n - 2$, then for every $\beta \in (0, \frac{2-\alpha}{n-p})$, we have
                \[
                |u(x)| \leq C(n, \Omega, \alpha, M, p, \beta) \dist(x,\partial\Omega)^\beta.
                \]

        \item[(iv)] If $n(\alpha - 1) - 2 < p-\alpha < n - 2$, then $u$ has the global H\"older estimates
                \[
                        |u(x)| \leq C(n, \Omega, \alpha, M, p) \dist(x,\partial\Omega)^{\frac{2-\alpha}{n-p}}.
                \]
\end{itemize}
\end{thm}

\begin{rems}
Some remarks on Theorem~\ref{thm:intro-Lip-alpha} are in order.
\begin{itemize}
        \item When $p\neq n$,  $\|u\|_{L^\infty(\Omega)}$ can be absorbed into the constant in Theorem \ref{thm:intro-Lip-alpha} (i). It is only necessary in the scale-invariant case $p=n$; see Proposition \ref{unibd}.
        \item The additional condition $p-\alpha>n(\alpha-1)-2$
in part~(iv) guarantees finite energy on every bounded convex domain, as required for the comparison argument in Section~\ref{sec:refined-estimates}; see Theorem~\ref{thm:intro-energy-estimates}. Also, in part~(ii), since $p - \alpha = n-2$, this condition is automatically satisfied. 

        \item One might wonder if the conclusion of Theorem \ref{thm:intro-Lip-alpha} (iv) is still valid when $p-\alpha \leq n(\alpha - 1) - 2$.
        \item For $\alpha=0$, the log-Lipschitz estimate (with correction of order $\frac{n}{2-\alpha}$) in part~(ii) and the endpoint estimate in part~(iv) recover \cite[Theorem 1.5(i) and Proposition 2.2]{Le_EVP}, respectively.
        For $\alpha>0$, they give the corresponding log-Lipschitz and H\"older estimates for the distance-weighted measure in the stated range.

\end{itemize}
\end{rems}

We prove parts~(i) and~(iii) of Theorem~\ref{thm:intro-Lip-alpha} in Section~\ref{sec:boundary-regularity}, and parts~(ii) and~(iv) in Section~\ref{sec:refined-estimates}.

\medskip
After normalizing $M=1$, we show a lower estimate near a flat part of the boundary.  

\begin{thm}[Growth of solutions near a flat boundary]\label{thm:intro-inf-bd}\label{thm: log_lower}\label{thm: ubd_grad}
Let $\Omega\subset\mathbb{R}^n (n\geq 2)$ be a bounded convex domain. Let
$0 \leq \alpha<2$ and $p>0$. Suppose that
$u\in C(\overline\Omega)$ is a nonzero convex function satisfying
\[
\mu_u\geq  |u|^p \dist(\cdot,\partial\Omega)^{-\alpha}\,d\mathcal{L}^n
\quad\text{in }\Omega,
\qquad
u=0\quad\text{on }\partial\Omega.
\]
Assume that there is a closed subset $\Gamma\subset\partial\Omega$ lying in a hyperplane and containing an $(n-1)$-dimensional ball of positive radius. 
Let $\Gamma' \subset \Gamma$ be a nonempty compact set in the interior of $\Gamma$. 
Then, there exists a constant $c = c(n,p,\alpha,\Omega, \Gamma, \Gamma')>0$ such that the following holds for $x\in\Omega$ sufficiently close to $\Gamma'$:
\begin{itemize}
    \item[(i)] If $p-\alpha<n-2$, then
            \begin{equation}\label{eq:power-conclusion}
                |u(x)|\geq c\,\dist(x,\partial\Omega)^{\frac{2-\alpha}{n-p}}.
            \end{equation}

    \item[(ii)] If $p-\alpha=n-2$, then
            \begin{equation}\label{eq:log-conclusion}
                |u(x)|\geq c\,\dist(x,\partial\Omega)
            \left|\log \dist(x,\partial\Omega)\right|^{\frac{n}{2-\alpha}}.
            \end{equation}
\end{itemize}
\end{thm}

\begin{rems}
        We now compare Theorem \ref{thm:intro-inf-bd} with related literature.
        \begin{itemize}
                \item In view of Theorem \ref{thm:intro-Lip-alpha} and Theorem \ref{thm:intro-inf-bd}, when $n(\alpha-1)-2<p-\alpha<n-2$, the upper and lower bounds have the same power $(2-\alpha)/(n-p)$. 
                        At $p-\alpha=n-2$, they have the same exponent $n/(2-\alpha)$ of $|\log \dist(x,\partial\Omega)|$. Thus the Lipschitz threshold is sharp on domains with a flat boundary part. 


                \item When $\alpha = 0$ and $p < n-2$, the lower bound in part (i) improves the lower bound in \cite[Theorem 1.1]{Le_23} up to the endpoint $\frac{2}{n-p}$.

                \item When $\alpha = 0$, the log-Lipschitz lower estimate (with correction of order $\frac{n}{2-\alpha}$) in Theorem~\ref{thm:intro-inf-bd}(ii) is new when $n\geq 3$. 
                For $n=2$, we have the critical exponent $p=n-2=0$.
On the triangle $T$ with vertices $(0,0)$, $(1,0)$, and $(0,1)$,
the problem
\[
\det D^2\sigma_T=1\quad\text{in }T,
\qquad \sigma_T=0\quad\text{on }\partial T
\]
admits an explicit convex solution known as the {\it surface tension}.
See Cohn--Kenyon--Propp \cite{CKP},
Kenyon--Okounkov \cite{KO},
Kenyon--Okounkov--Sheffield \cite{KOS},
Astala--Duse--Prause--Zhong \cite{ADPZ},
Mikhalkin--Rullg{\aa}rd \cite{MR}, and Li \cite{Li};
see also \cite[Example 3.32]{Le_24}.
Near the interior of each side of the boundary, this solution grows at the rate
$\dist(\cdot,\partial T)|\log\dist(\cdot,\partial T)|$,
with logarithmic exponent $1=n/2$.

For $n\geq3$ and $p=n-2$, \eqref{eq: leb_eq}
 has the degenerate right-hand side $|u|^{n-2}$. Theorem~\ref{thm:intro-inf-bd}(ii)
gives the logarithmic exponent $n/2$ in the lower estimate near
flat boundary parts of general bounded convex domains.
This improves the exponent $1/n$ in \eqref{Le_n-2_low}
and matches the upper estimate \eqref{Le_n-2_up},
answering the question raised by Le
\cite[after Theorem 1.5]{Le_EVP}.
        \end{itemize}
\end{rems}

We prove Theorem~\ref{thm:intro-inf-bd} in Section~\ref{sec:flat-boundary-estimates}. 

\medskip
Our final results are concerned with the Monge--Amp\`ere energy and the Monge--Amp\`ere eigenvalue. For a solution to \eqref{main_eq},
\[
I[u]=M\int_\Omega |u|^{p+1}\dist(\cdot,\partial\Omega)^{-\alpha}\,dx,
\]
so we can use the boundary estimates to determine when $u$ has finite energy.

\begin{thm}[Finite energy threshold]\label{thm:intro-energy-estimates}
Let $\Omega \subset \R^n (n \geq 2)$ be a bounded convex domain, $0 \leq \alpha < 2$, $p>0$, and $M>0$. Let $u \in C(\overline{\Omega})$ be a nonzero convex solution to
\[
\mu_u=M|u|^p \dist(\cdot,\partial\Omega)^{-\alpha}\,d\mathcal{L}^n\quad\text{in }\Omega,
\qquad u=0\quad\text{on }\partial\Omega.
\]
\begin{itemize}
  \item[(i)] If $p - \alpha > (\alpha - 1)n - 2$, then $I[u] < \infty$.
  \item[(ii)] If $p - \alpha \leq (\alpha - 1)n - 2$ and there is a closed subset $\Gamma\subset\partial\Omega$ lying in a hyperplane and containing an $(n-1)$-dimensional ball of positive radius, then $I[u] = \infty$.
\end{itemize}
\end{thm}

\begin{rem}\label{rem:weighted-uniqueness}
For $0<p<n$, $0\leq\alpha<2$, and $p-\alpha>n(\alpha-1)-2$, Theorems~\ref{thm: p<n_uniq} and~\ref{thm:intro-energy-estimates}(i) imply the uniqueness of nonzero convex solutions to \eqref{main_eq} for each fixed $M>0$.
\end{rem}

We prove Theorem~\ref{thm:intro-energy-estimates} in Section~\ref{sec:energy-estimates}.

\medskip
The Lipschitz estimate in Theorem~\ref{thm:intro-Lip-alpha}(i), together with \cite[Theorem 6.9, Corollary 6.10, and Theorem 6.14]{Le_Var}, gives the following result for the Monge--Amp\`ere eigenvalue.

\begin{cor}[Variational characterization and uniqueness of the Monge--Amp\`ere eigenvalue]\label{cor:intro-eigenvalue}
Let $\Omega\subset\mathbb R^n$, $n\geq2$, be a bounded convex domain and $0\leq\alpha<2$. Suppose that $\lambda>0$ and $u\in C(\overline\Omega)$ is a nonzero convex solution to the Monge--Amp\`ere eigenvalue problem
\[
\mu_u=\lambda|u|^n\dist(\cdot,\partial\Omega)^{-\alpha}\,d\mathcal L^n
\quad\text{in }\Omega,
\qquad u=0\quad\text{on }\partial\Omega.
\]
Then $\lambda$ is given by the variational characterization
\[
\lambda=
\inf\left\{
\frac{\displaystyle\int_\Omega |w|\,d\mu_w}
{\displaystyle\int_\Omega |w|^{n+1}\dist(\cdot,\partial\Omega)^{-\alpha}\,dx}
:
\begin{array}{l}
w\in C^{0,1}(\overline\Omega)\setminus\{0\},\\
w\text{ is convex in }\Omega,\quad w=0\text{ on }\partial\Omega
\end{array}
\right\}.
\]
Moreover, if $v\in C(\overline\Omega)$ is a nonzero convex function satisfying
\[
\mu_v\geq\Lambda|v|^n\dist(\cdot,\partial\Omega)^{-\alpha}\,d\mathcal L^n
\quad\text{in }\Omega,
\qquad v=0\quad\text{on }\partial\Omega,
\]
for some $\Lambda>0$, then $\lambda\geq\Lambda$.
If the convex function $v \in C(\overline{\Omega})\setminus\{0\}$ solves the eigenvalue problem with eigenvalue $\Lambda$, then $\lambda=\Lambda$ and $u=cv$ for some constant $c>0$.
\end{cor}

When $p=n$, Theorem~\ref{thm:intro-Lip-alpha}(i) gives $u\in C^{0,1}(\overline\Omega)$. Every convex Lipschitz function with zero boundary values has finite energy, so \cite[Theorem 6.9]{Le_Var} gives the variational characterization over $\mathbb E(\Omega)\setminus\{0\}$. Since $u$ itself attains the Rayleigh infimum, restricting the infimum to convex Lipschitz functions does not change its value. The comparison and uniqueness statements follow from \cite[Corollary 6.10 and Theorem 6.14]{Le_Var}.

\subsection{Ingredients of the proofs}
We say a few words on the proofs of the main results.
Our proof of Theorem~\ref{thm: p<n_uniq} is inspired by Le's existence result \cite[Theorem 6.4]{Le_Var} and its proof. We use the nonzero solution $u^*$ obtained there by truncating the measure away from the boundary. Our key observation is that $u^*$ lies above every nonzero convex subsolution with nonpositive boundary values. Thus any nonzero finite-energy solution $u$ satisfies $u\leq u^*$, and monotonicity of the energy class gives $u^*\in\mathbb E(\Omega)$. This allows us to combine integration by parts with the mixed Monge--Amp\`ere inequality to obtain $u=u^*$ $\nu$-almost everywhere. The two solutions therefore have the same Monge--Amp\`ere measure, and the classical comparison principle gives $u=u^*$ in $\Omega$.

\medskip
For Theorem~\ref{thm:supercritical-two-shell}, we use a piecewise affine radial profile whose Monge--Amp\`ere measure is concentrated on two spherical shells. The equations for the masses on the two shells reduce to a one-dimensional equation. We choose the radii and masses so that this equation has three roots, giving three distinct solutions for the same measure supported in the interior.

\medskip
For the upper estimates in Theorem~\ref{thm:intro-Lip-alpha}, the uniform Aleksandrov--Jerison maximum principle, Theorem~\ref{AJ_thm}, gives an initial H\"older estimate. We improve this estimate iteratively by comparing with solutions to equations of the form $\mu_w=A|w|^q\,d\mathcal{L}^n$ and applying Le's boundary estimates for the Lebesgue measure case \cite{Le_22,Le_EVP}. This proves Theorem~\ref{thm: Lip_alpha}: when $p-\alpha>n-2$, the iteration gives a Lipschitz estimate; otherwise it gives every exponent $0<\beta<(2-\alpha)/(n-p)$. In the range $n(\alpha-1)-2<p-\alpha\leq n-2$, these bounds imply finite energy. The uniqueness of nonzero finite-energy solutions and Lemma~\ref{lem: maximal} then give the comparison with a global convex subsolution in Proposition~\ref{prop: restricted_comp}.

\medskip
For the endpoint estimates in Theorem~\ref{thm:intro-Lip-alpha}, we adapt Le's subsolutions for the Hölder and log-Lipschitz estimates \cite[Proposition 2.2 and Lemma 4.1]{Le_EVP}. His constructions use a fixed supporting hyperplane of $\Omega$ and give convex functions that are nonpositive on $\partial\Omega$. In Lemma~\ref{lem: refined_subsoln}, we take the supremum of the corresponding functions over all supporting hyperplanes to obtain a convex subsolution with zero boundary values. For a function attaining the supremum at $x$, the distance from $x$ to the corresponding supporting hyperplane is comparable to $\dist(x,\partial\Omega)$. This comparison gives the required measure inequality with weight $\dist(\cdot,\partial\Omega)^{-\alpha}$.

\medskip
For Theorem~\ref{thm:intro-inf-bd}, we estimate the Monge--Amp\`ere mass of a small cylinder inside a cone near a flat part of the boundary. Subtracting a supporting plane gives an upper bound for this mass, while the measure inequality gives a lower bound. Comparing the two bounds yields a differential inequality in terms of the solution in the normal direction, as stated in Lemma~\ref{lem:cone-mass}. Then, integration gives the power and log-Lipschitz lower estimates.

\medskip
The upper estimates in Theorem~\ref{thm: Lip_alpha} imply the sufficient condition for finite energy in Theorem~\ref{thm:intro-energy-estimates}. The lower bounds in Theorem~\ref{thm:intro-inf-bd} give necessity when the boundary contains a flat part. For $p=n$, the Lipschitz estimate allows us to apply Le's variational characterization, comparison, and uniqueness results to obtain Corollary~\ref{cor:intro-eigenvalue}.

\medskip
\noindent\textbf{Organization of the paper.} The rest of the paper is organized as follows. 
Section~\ref{sec:preliminaries} recalls the comparison principles and energy monotonicity, together with the mixed Monge--Amp\`ere measures and their properties.
In Section~\ref{sec:general-measures}, we prove the uniqueness of nonzero convex solutions in the finite-energy class for $0<p<n$ and construct the two-shell example of nonuniqueness for $p>n$.
In Section~\ref{sec:weighted-measure}, we first establish the Lipschitz and H\"older upper estimates for the distance-weighted measure. We then prove the lower estimates near a flat part of the boundary and derive the energy estimates.
In Section~\ref{sec:refined-estimates}, the finite energy and the uniqueness of nonzero solutions allow us to compare solutions with convex subsolutions. We construct these subsolutions and use them to prove the endpoint H\"older and log-Lipschitz upper bounds in Theorem~\ref{thm:intro-Lip-alpha}.

\medskip
\noindent\textbf{Acknowledgments.} The author would like to thank his advisor, Professor Nam Q. Le, for suggesting the problems,
and for his patience, guidance, and encouragement throughout the preparation of this work.

\medskip
\noindent\textbf{AI assistance.} The author used ChatGPT 6 Astra for calculations in Theorem \ref{thm:supercritical-two-shell}, Lemma \ref{lem:cone-mass}, and Lemma \ref{lem: refined_subsoln} and language improvements. 
The key ideas and proofs are from the author. The author checked the calculations and takes full responsibility for their correctness.

\section{Comparison principles and mixed Monge--Amp\`ere measures}
\label{sec:preliminaries}

We use the notions of subdifferential, Aleksandrov solution, finite-energy class, and Monge--Amp\`ere energy introduced in Section~\ref{sec:introduction}. 
For further background on Aleksandrov solutions, we refer to \cite[Chapter 3]{Le_24}.

We recall the following classical comparison principle.
\begin{thm}[Comparison principle {\cite[Theorem 3.21]{Le_24}}]\label{thm:classical-comparison}
Let $\Omega\subset\R^n$ be a bounded domain and let $u,v\in C(\overline\Omega)$ be convex. If
\[
\mu_u\leq\mu_v\quad\text{in }\Omega,
\qquad u\geq v\quad\text{on }\partial\Omega,
\]
then $u\geq v$ in $\Omega$.
\end{thm}
In particular, equal Monge--Amp\`ere measures and equal boundary values determine a convex solution uniquely.

\medskip
We also use the following comparison principle for subcritical Monge--Amp\`ere equations in the setting of general Borel measures. 
\begin{thm}[Subcritical comparison principle {\cite[Theorem 1.5]{Le_Var}}]\label{thm:subcritical-comparison}
Let $0<p<n$, let $\Omega\subset\R^n$ be a bounded convex domain, and let $\nu$ be a Borel measure in $\Omega$. Assume that either
\[
\nu=f\,d\mathcal{L}^n,\qquad f\in C(\overline\Omega),\qquad f>0\quad\text{in }\overline\Omega,
\]
or that $\nu$ is compactly supported in $\Omega$. Suppose that $u,v\in C(\overline\Omega)$ are convex,
\[
v<0\quad\text{in }\Omega,\qquad v\leq0\quad\text{on }\partial\Omega,
\qquad \mu_v\geq |v|^p\nu\quad\text{in }\Omega,
\]
and
\[
u=0\quad\text{on }\partial\Omega,
\qquad \mu_u\leq |u|^p\nu\quad\text{in }\Omega.
\]
Then $u\geq v$ in $\Omega$, and consequently $\mu_u\leq\mu_v$ in $\Omega$.
\end{thm}
Applying this theorem in both directions gives the uniqueness of nonzero convex solutions.
We apply the compactly supported case to truncated measures to establish the maximality property in Lemma~\ref{lem: maximal}.

We also need the following monotonicity property of the energy class.
\begin{prop}[Monotonicity of the energy class {\cite[Proposition 4.8]{Le_Var}}]\label{prop:energy-monotonicity}
Let $\Omega\subset\R^n$ be a bounded convex domain, $w\in\mathbb E(\Omega)$, and let $\widetilde w\in C(\overline\Omega)$ be convex. If $\widetilde w=0$ on $\partial\Omega$ and $w\leq\widetilde w$ in $\Omega$, then $\widetilde w\in\mathbb E(\Omega)$ and
\[
I[\widetilde w]\leq I[w].
\]
\end{prop}

\medskip
Next, we discuss the mixed Monge--Amp\`ere measure and its properties. The definition and results below are from \cite{Le_Var}.
\begin{defn}[Mixed Monge--Amp\`ere measure {\cite[Definition 1.6]{Le_Var}}]\label{def:mixed-ma}
For convex functions $u_1,\ldots,u_n$ on $\Omega\subset\R^n$, define
\[
\mu_n[u_1,\ldots,u_n]
:=\frac1{n!}\sum_{k=1}^n(-1)^{n-k}
\sum_{1\leq i_1<\cdots<i_k\leq n}
\mu_{u_{i_1}+\cdots+u_{i_k}}.
\]
The map $(u_1,\ldots,u_n)\mapsto\mu_n[u_1,\ldots,u_n]$ is symmetric and multilinear with respect to positive linear combinations, and its values are nonnegative Borel measures. Moreover,
\[
\mu_n[u,\ldots,u]=\mu_u.
\]
\end{defn}

\begin{thm}[Integration by parts {\cite[Theorem 1.8(ii)]{Le_Var}}]\label{thm:mixed-ibp}
Let $\Omega\subset\R^n$ be a bounded convex domain and let $u_0,\ldots,u_n\in\mathbb E(\Omega)$. Then
\[
\int_\Omega u_0\,d\mu_n[u_1,\ldots,u_n]
=
\int_\Omega u_n\,d\mu_n[u_0,u_1,\ldots,u_{n-1}].
\]
\end{thm}

\begin{thm}[Mixed Monge--Amp\`ere inequality {\cite[Theorem 1.9]{Le_Var}}]\label{thm:mixed-ma-inequality}
Let $\Omega\subset\R^n$ be a bounded convex domain, let $\nu$ be a Borel measure in $\Omega$, and let $u_1,\ldots,u_n\in C(\Omega)$ be convex. Suppose that, for $1\leq i\leq n$,
\[
0\leq f_i\in L^1_{\mathrm{loc}}(\Omega,d\nu),
\qquad
\mu_{u_i}\geq f_i\nu.
\]
Then
\[
\mu_n[u_1,\ldots,u_n]
\geq \left(\prod_{i=1}^n f_i^{1/n}\right)\nu.
\]
\end{thm}


\section{Uniqueness and nonuniqueness for general Borel measures}
\label{sec:general-measures}
In this section, we prove Theorems~\ref{thm: p<n_uniq} and~\ref{thm:supercritical-two-shell}.

\medskip
For $0<p<n$, Le \cite[Theorem 6.4]{Le_Var} proved the existence of a nonzero continuous Aleksandrov solution under the measure assumptions of Theorem~\ref{thm: p<n_uniq}; we prove the uniqueness of nonzero finite-energy solutions. 
For $p>n$, we construct a compactly supported measure admitting at least three distinct radial solutions, all globally Lipschitz and of finite energy. 
Since $p\ne n$ throughout this section, we let $M=1$ in \eqref{eq:intro-general}.


\subsection[Uniqueness of solutions for p in (0,n)]{Uniqueness of solutions for \texorpdfstring{$0 < p < n$}{p in (0,n)}}\label{sec:subcritical-uniqueness}
We first establish the maximality property of the solution constructed by truncation in \cite[Theorem 6.4]{Le_Var}.
\begin{lem}[Maximal subsolution]\label{lem: maximal}
Let $\Omega \subset \R^n (n \geq 2)$ be a bounded convex domain.
Let $p \in (0, n)$ and $\nu$ be a locally finite Borel measure on $\Omega$ with 
\[
\nu (\Omega) > 0 \quad \text{and} \quad \int_\Omega \dist(\cdot, \partial \Omega) \; d\nu < \infty.
\]
Then, there exists a nonzero convex function $u^* \in C(\overline{\Omega})$ satisfying
\begin{equation}\label{eq: gen_u*}
\mu_{u^*} = |u^*|^p \nu \quad \text{in } \Omega, \quad \text{and } u^* = 0 \quad \text{on } \partial \Omega,
\end{equation}
such that for any convex function $u \in C(\overline{\Omega})\setminus \{0\}$ satisfying 
\begin{equation}\label{eq: gen_sub}
\mu_u \geq |u|^p \nu \quad \text{in } \Omega, \quad \text{and } u \leq 0 \quad \text{on } \partial \Omega,
\end{equation}
we have $u \leq u^* \leq  0$ in $\overline{\Omega}$.
\end{lem}

\begin{proof}
For $\e > 0$, let $\Omega_\e := \{x \in \Omega: \dist(x,\partial\Omega) > \e\}$ and $\nu_\e := \chi_{\Omega_\e} \nu$.
We take $\e$ sufficiently small so that $\nu_\e(\Omega)>0$. Then,
\[
\int_\Omega \dist(\cdot,\partial\Omega) \; d\nu_\e \leq \int_\Omega \dist(\cdot,\partial\Omega) \; d\nu < \infty.
\]
By \cite[Theorem 6.4]{Le_Var}, there exists a nonzero convex function $w_\e \in C(\overline{\Omega})$ such that
\begin{equation}\label{eq: gen_truncate}
\mu_{w_\e} = |w_\e|^p \nu_\e \quad \text{in } \Omega, \quad \text{and } w_\e = 0 \quad \text{on } \partial \Omega.
\end{equation}
Moreover, since $\nu_\e$ is compactly supported in $\Omega$, $w_\e$ is the unique nonzero convex solution by Theorem~\ref{thm:subcritical-comparison}.
The proof of \cite[Theorem 6.4]{Le_Var} shows that, by letting $\e \to 0$, we obtain a nonzero convex function $u^* \in C(\overline{\Omega})$ satisfying \eqref{eq: gen_u*}.

Now fix a convex subsolution $u \in C(\overline{\Omega})\setminus\{0\}$ satisfying \eqref{eq: gen_sub}.
Note that  $\mu_u \geq |u|^p \nu \geq |u|^p \nu_\e$. Hence, $u$ is a subsolution to \eqref{eq: gen_truncate}.
Since $u \leq w_\e = 0$ on $\partial \Omega$, Theorem~\ref{thm:subcritical-comparison} gives $u \leq w_\e\leq 0$ in $\overline{\Omega}$.
Therefore, by passing to the limit, we also have $u \leq u^* \leq 0$ in $\overline{\Omega}$. The proof is complete.
\end{proof}

We now prove Theorem~\ref{thm: p<n_uniq} by combining Lemma~\ref{lem: maximal} with the mixed Monge--Amp\`ere argument used in the proof of \cite[Lemma 6.5]{Le_Var}.

\begin{proof}[Proof of Theorem \ref{thm: p<n_uniq}]
By Lemma \ref{lem: maximal}, we can find a nonzero convex function $u^* \in C(\overline{\Omega})$ such that $u \leq u^* \leq 0$ in $\overline{\Omega}$, and
\[
\mu_{u^*} = |u^*|^p \nu \quad \text{in } \Omega, \quad \text{and } u^* = 0 \quad \text{on } \partial \Omega. 
\]
Since $u \in \mathbb E(\Omega)$ and $u\leq u^*\leq0$, Proposition~\ref{prop:energy-monotonicity} gives $u^* \in \mathbb E(\Omega)$. We also have
\[
\int_\Omega |u|^p|u^*|\,d\nu
\leq \int_\Omega |u|^{p+1}\,d\nu
=I[u]<\infty.
\]
We can therefore apply Theorem~\ref{thm:mixed-ibp} to obtain
\begin{equation}\label{uniq_ineq1}
\begin{aligned}
\int_\Omega |u|^p |u^*| \; d\nu = \int_\Omega |u^*| \; d \mu_u &= \int_\Omega |u^*|\; d\mu_n[u, \ldots, u]\\
    & = \int_\Omega |u| \; d\mu_n[u^*, u, \ldots, u].
\end{aligned}
\end{equation}
Theorem~\ref{thm:mixed-ma-inequality} shows that
\begin{equation}\label{uniq_ineq2}
\mu_n[u^*, u, \ldots, u] \geq \left( |u|^p\right)^{\frac{n-1}{n}} \left( |u^*|^p\right)^{\frac{1}{n}} \nu = |u|^{\frac{np-p}{n}} |u^*|^{\frac{p}{n}} \nu.
\end{equation}
Combining \eqref{uniq_ineq1} and \eqref{uniq_ineq2} gives
\begin{equation}\label{uniq_ineq3}
\int_\Omega |u|^p |u^*| \; d\nu \geq \int_\Omega |u|^{p+1 - \frac{p}{n}} |u^*|^{\frac{p}{n}} \; d\nu.
\end{equation}
Since $p < n$ and $|u| \geq |u^*|$, we have
\[
|u|^p|u^*|\leq |u|^{p+1 - \frac{p}{n}} |u^*|^{\frac{p}{n}} \quad \text{in } \Omega.
\]
Together with \eqref{uniq_ineq3}, we must have $|u| = |u^*|$ $\nu$-a.e., and hence $\mu_u = \mu_{u^*}$. 
Since $u = u^* = 0$ on $\partial \Omega$, Theorem~\ref{thm:classical-comparison} gives $u = u^*$ on $\Omega$.

Now, by repeating the same process for $v$ and the same $u^*$, we get $v = u^*$ on $\Omega$. Therefore, $u = u^* = v$ on $\Omega$. The proof is complete.

\end{proof}


\subsection[Nonuniqueness of solutions for p greater than n]{Nonuniqueness of solutions for \texorpdfstring{$p>n$}{p greater than n}}\label{sec:supercritical-nonuniqueness}
In this subsection, we construct the measure $\nu$ on the unit ball and the solutions $u_1,u_2,u_3$ in Theorem \ref{thm:supercritical-two-shell}.

We begin with a family of radial convex functions whose Monge--Amp\`ere measures are supported on two spherical shells. The parameters will be fixed later.

\medskip
\begin{lem}
\label{lem:two-shell-candidate}
Let $0<r_1<r_2<1$.
For $s>0$ and $t\in(0,1)$, define
\[
 u_{s,t}(x):=U_{s,t}(|x|), \quad x \in \overline{B_1},
\]
where
\[
 U_{s,t}(r):=
 \begin{cases}
  -s(1 - r_2) - st(r_2 - r_1) ,&-1\leq r\leq r_1,\\[1mm]
  -s(1 - r_2) + st(r-r_2),&r_1\leq r\leq r_2,\\[1mm]
  s(r-1),&r_2\leq r\leq 1.
 \end{cases}
\]
Then
\(u_{s,t}\in C(\overline{B_1})\) is convex and satisfies
\[
 u_{s,t}<0\quad\text{in }B_1,
 \qquad
 u_{s,t}=0\quad\text{on }\partial B_1,
\]
and
\begin{equation}\label{eq: u_measure}
 \mu_{u_{s,t}}
 =\omega_ns^n\bigl[t^n\sigma_{r_1}+(1-t^n)\sigma_{r_2}\bigr].
\end{equation}
\end{lem}

\begin{proof}
Notice that $U_{s,t}$ is piecewise affine with strictly increasing slopes. 
Moreover, $U_{s,t}$ is continuous, convex, and nondecreasing on $[0, 1]$.
Hence, $u_{s,t} \in C(\overline{B_1})$ and is convex in $B_1$.
Furthermore, we have 
\[
 u_{s,t}<0\quad\text{in }B_1,
 \qquad \text{and} \qquad
 u_{s,t}=0\quad\text{on }\partial B_1.
\]

\medskip
Next, we calculate the subdifferential of $u_{s, t}$.
We first compute the subdifferential of $U_{s, t}$. Indeed, it is easy to see by definition that
\begin{equation}\label{eq: U_subd}
\partial U_{s, t} (r) = 
\begin{cases}
  \{0\},  &0 \leq r <r_1,\\[1mm]
  [0, st], & r=r_1,\\[1mm]
  \{st\}, &r_1<r<r_2,\\[1mm]
  [st, s],  &r=r_2,\\[1mm]
  \{s\},  &r_2<r<1.
 \end{cases}
\end{equation}

At the origin, $\partial u_{s,t}(0)=\{0\}$ because $u_{s,t}$ is constant in $B_{r_1}$.
Now, let $x = re$ with $r \in (0, 1)$ and $e \in \mathbb{S}^{n-1} = \partial B_1(0)$ being a unit vector. Then
\[
\partial u_{s, t}(re) = \{\rho e:\rho\in\partial U_{s, t}(r)\}.
\]
Indeed, if $q\in\partial u_{s,t}(re)$, then for every $\theta\in\mathbb S^{n-1}$, 
\begin{equation}\label{eq:two_shell_1}
u_{s, t}(r\theta) \geq u_{s, t}(re) + q\cdot (r\theta - re).
\end{equation}
Hence, $u_{s,t}(r\theta) = U_{s, t}(r) =u_{s,t}(re)$ gives $q\cdot\theta\leq q\cdot e$.
Taking the supremum over $\theta$ gives $|q|\leq q\cdot e\leq |q|$, so $q=\rho e$ with some $\rho\geq0$.
Then, notice that $q\in\partial u_{s,t}(re)$ also gives that for $r'\in[0,1)$,
\[
U_{s, t}(r') = u_{s, t}(r'e) \geq u_{s, t}(re) + q\cdot (r'e - re) = U_{s ,t}(r) + \rho (r' - r),
\]
which shows $\rho\in\partial U_{s,t}(r)$.

Conversely, if $\rho\in\partial U_{s,t}(r)$, then $\rho\geq0$, and for every $y\in B_1$,
\[
u_{s,t}(y)=U_{s,t}(|y|)\geq U_{s,t}(r)+\rho(|y|-r)
\geq u_{s,t}(re)+\rho e\cdot(y-re).
\]
Thus $\rho e\in\partial u_{s,t}(re)$.

\medskip

Therefore, recalling \eqref{eq: U_subd} shows that the subdifferential of $u_{s, t}$ is 
\[
 \partial u_{s,t}(x)=
 \begin{cases}
  \{0\},&|x|<r_1,\\[1mm]
  \{\rho \frac{x}{r_1}:0\leq\rho\leq st\},&|x|=r_1,\\[1mm]
  \{st\frac{x}{|x|}\},&r_1<|x|<r_2,\\[1mm]
  \{\rho \frac{x}{r_2}:st\leq\rho\leq s\},&|x|=r_2,\\[1mm]
  \{s\frac{x}{|x|}\},&r_2<|x|<1.
 \end{cases}
\]
In particular, for $x \in B_1$ with $|x| < r_1, r_1 < |x| < r_2$, or $r_2 < |x| < 1$, 
$\partial u_{s, t}(x)$ is contained in $\{0\}, \partial B_{st}$, or $\partial B_s$, respectively, 
which all have zero $n$-dimensional Lebesgue measure. Therefore, 
\begin{equation}\label{eq: u_1}
 \supp\mu_{u_{s,t}}\subset \partial B_{r_1}\cup \partial B_{r_2}.
\end{equation}



\medskip
Finally, we prove \eqref{eq: u_measure}.
Let $E \subset B_1$ be a Borel set and let $E_1 = E \cap \partial B_{r_1}$ and $E_2 = E \cap \partial B_{r_2}$.
Then, by \eqref{eq: u_1}, we have
\begin{equation}\label{eq: u_2}
\mu_{u_{s, t}}(E) = |\partial u_{s, t}(E_1)| + |\partial u_{s, t}(E_2)|.
\end{equation}
To compute $|\partial u_{s, t}(E_1)|$, we use polar coordinates and \eqref{eq:normalized-sphere-measure}:
\begin{equation}\label{eq: u_3}
\begin{aligned}
 |\partial u_{s,t}(E_1)|
 &=\int_{\{e\in \mathbb{S}^{n-1}: r_1e \in E_1\}}\int_0^{st}\rho^{n-1}\,d\rho\,
 d\mathcal{H}^{n-1}(e)\\
 &=\frac{s^nt^n}{n}\mathcal{H}^{n-1}(\{e\in \mathbb{S}^{n-1}: r_1e \in E_1\})\\
 &=\omega_ns^nt^n\sigma_{r_1}(E_1).
\end{aligned}
\end{equation}
Similarly,
\begin{equation}
 |\partial u_{s,t}(E_2)|
 =\omega_ns^n(1-t^n)\sigma_{r_2}(E_2).
 \label{eq: u_4}
\end{equation}
Combining \eqref{eq: u_1}--\eqref{eq: u_4}
yields
\[
 \mu_{u_{s,t}}
 =\omega_ns^n\bigl[t^n\sigma_{r_1}+(1-t^n)\sigma_{r_2}\bigr],
\]
which completes the proof.
\end{proof}

We next choose $r_1$, $r_2$, and $\kappa>0$ so that the equation $F_{r_1,r_2}(t)=\kappa$ in the following lemma has three distinct roots. The proof is an elementary calculus argument.
\begin{lem}
\label{lem:three-root-property}
Let $p>n\geq2$.  There exist constants $0<r_1<r_2<1$ and $\kappa>0$, depending only on $n$ and $p$, 
such that
\[
 F_{r_1,r_2}(t)
 :=\frac{t^n}{1-t^n}\left(\frac{1- r_2}{1-r_2+ (r_2 - r_1) t}\right)^p,
 \qquad 0<t<1,
\]
satisfies $F_{r_1,r_2}(t)=\kappa$ at exactly three distinct points of
$(0,1)$.
\end{lem}

\begin{proof}
Let $b:=1-r_2$ and $\ell:=r_2-r_1$.
Differentiating $F_{r_1, r_2}(t)$ gives
\begin{equation}\label{eq: DF}
F'_{r_1, r_2}(t) = \frac{b^pt^{n-1}}{(1-t^n)^2(b+\ell t)^{p+1}}G_{r_1, r_2}(t),
\end{equation}
where
\[
 G_{r_1, r_2}(t):=nb-(p-n)\ell t+p\ell t^{n+1}.
\]
Since $b,\ell>0$ and $0<t<1$, $F'_{r_1,r_2}(t)$ has the same sign as $G_{r_1,r_2}(t)$.
Since
\[
 G_{r_1, r_2}'(t)
 =\ell\bigl[p(n+1)t^n-(p-n)\bigr],
\]
$G'_{r_1, r_2}$ has a unique zero at 
\[
 t_*:=\left(\frac{p-n}{p(n+1)}\right)^{1/n}\in(0,1).
\]
Moreover, $G_{r_1 ,r_2}' < 0$ if $0 < t < t_*$, and $G_{r_1 ,r_2}' > 0$ if $t_* < t < 1$.

Let
\[
K:=\frac{2(n+1)}{(p-n)t_*} > 0.
\]
Choose
$r_1:=\frac12$ and
$r_2:=1-\frac{1}{2(K+1)}$.
Then $0<r_1<r_2<1$, and
\[
b=1-r_2=\frac{1}{2(K+1)},
\qquad
\ell=r_2-r_1=\frac{K}{2(K+1)}.
\]
Therefore, by recalling the definitions of $t_*$ and $K$,
we have
\begin{align*}
G_{r_1, r_2}(t_*)
& = nb - (p-n)\ell t_* + p\ell t_*^{n+1}\\
&=
\frac{n}{2(K+1)}
-\frac{(p-n)K}{2(K+1)}t_*
+\frac{(p-n)K}{2(n+1)(K+1)}t_*\\
&=
\frac{n}{2(K+1)}
\left(
1-\frac{(p-n)Kt_*}{n+1}
\right)\\
&=
-\frac{n}{2(K+1)}
<0.
\end{align*}
Moreover,
\[
G_{r_1, r_2}(0)=nb>0,
\qquad
G_{r_1, r_2}(1)=n(b+\ell)>0.
\]
Since \(G_{r_1, r_2}\) is strictly decreasing on \((0,t_*)\) and strictly
increasing on \((t_*,1)\), it follows from the intermediate value theorem that \(G_{r_1, r_2}\) has exactly two
zeros $0<\tau_-<t_*<\tau_+<1$.

\medskip
Consequently, by \eqref{eq: DF}, $F_{r_1,r_2}$ is strictly increasing on $(0,\tau_-)$,
strictly decreasing on $(\tau_-,\tau_+)$, and strictly increasing on
$(\tau_+,1)$. Hence, $F_{r_1,r_2}(\tau_+)<F_{r_1,r_2}(\tau_-)$. 
Moreover, since $\lim_{t\downarrow0}F_{r_1,r_2}(t)=0$ and $\lim_{t\uparrow1}F_{r_1,r_2}(t)=+\infty$,
we may take 
\[
\kappa = \frac{F_{r_1, r_2}(\tau_-) + F_{r_1, r_2}(\tau_+)}{2}.
\]
Then the equation $F_{r_1,r_2}(t)=\kappa$ has exactly one solution in each of
$(0,\tau_-)$, $(\tau_-,\tau_+)$, and $(\tau_+,1)$.  The proof is complete.
\end{proof}

We can now prove Theorem \ref{thm:supercritical-two-shell}.

\begin{proof}[Proof of Theorem \ref{thm:supercritical-two-shell}]
By Lemma \ref{lem:three-root-property}, there exist constants $0 < r_1 < r_2 < 1$ and $\kappa > 0$, depending only on $n$ and $p$, such that 
\[
 F_{r_1,r_2}(t)=\frac{t^n}{1-t^n}\left(\frac{1- r_2}{1-r_2+ (r_2 - r_1) t}\right)^p
\]
satisfies $F_{r_1,r_2}(t_i)=\kappa$ for $i=1,2,3$, where $0<t_1<t_2<t_3<1$.
Let $b=1-r_2$ and $\ell=r_2-r_1$. 

\medskip
Let $s_i = (1 - t_i^n)^{1/(p-n)}$ for $i = 1, 2, 3$. Define $u_{s_i, t_i}$ as in Lemma \ref{lem:two-shell-candidate}.
By choosing
\[
 A_2:=\frac{\omega_n}{b^p},
 \qquad
 A_1:=\kappa A_2,
\]
we will show that
\[
\mu_{u_{s_i, t_i}} = |u_{s_i, t_i}|^p\nu, \qquad i = 1, 2, 3.
\]
Indeed, 
since $s_i^{p-n}=1-t_i^n$,
\[
 A_2s_i^pb^p
 =\omega_ns_i^n(1-t_i^n),
\]
and 
\[
 A_1s_i^p(b+\ell t_i)^p
 =\omega_ns_i^nt_i^n.
\]
Hence, by Lemma \ref{lem:two-shell-candidate},
\[
\begin{aligned}
 \mu_{u_{s_i, t_i}}&= \omega_ns_i^n t_i^n\sigma_{r_1}+\omega_n s_i^n(1-t_i^n)\sigma_{r_2}\\
 & = s_i^p(b+\ell t_i)^p A_1\sigma_{r_1} + s_i^pb^p A_2\sigma_{r_2}\\
 &=|u_{s_i, t_i}|^p\nu
 \quad\text{in }B_1,
 \qquad i=1,2,3.
\end{aligned}
\]
Moreover, the function $t\mapsto(1-t^n)^{1/(p-n)}$ is strictly decreasing on
$(0,1)$, so $s_1>s_2>s_3$.  Since
\[
 u_{s_i, t_i}(x)=s_i(|x|-1)\qquad\text{when }r_2<|x|<1,
\]
the three solutions are pairwise distinct. Each $u_{s_i,t_i}$ is globally Lipschitz, since the slopes of $U_{s_i,t_i}$ are bounded by $s_i$. Moreover, \eqref{eq: u_measure} gives
\[
I[u_{s_i,t_i}]
=\omega_n s_i^{n+1}\bigl[b+\ell t_i^{n+1}\bigr]<\infty,
\qquad i=1,2,3.
\]
The proof is complete.

\end{proof}


\section[The distance-weighted measure]{Global H\"older estimates for the case of distance-weighted measures}
\label{sec:weighted-measure}
In this section, we study \eqref{main_eq}, which we restate for convenience:
\[
\left\{\begin{aligned}
\operatorname{det} D^2 u & = M |u|^p \dist(\cdot,\partial\Omega)^{-\alpha} & & \text { in } \Omega, \\
u & =0 & & \text { on }\partial \Omega.
\end{aligned}\right.
\]
Here $M>0$, $p>0$, and $0\leq \alpha<2$. 
We assume $M=1$ when $p\neq n$.

We first prove parts~(i) and~(iii) of Theorem~\ref{thm:intro-Lip-alpha}; see Theorem~\ref{thm: Lip_alpha}.
We then prove Theorems~\ref{thm:intro-inf-bd} and~\ref{thm:intro-energy-estimates} on flat boundary growth and finite energy, respectively.
The log-Lipschitz estimate in part~(ii) and the endpoint H\"older estimate in part~(iv) of Theorem~\ref{thm:intro-Lip-alpha} are proved in Section~\ref{sec:refined-estimates}.

\medskip
\subsection{Global regularity of solutions}\label{sec:boundary-regularity}
We begin with the maximum principle and a scale-invariant bound needed for the iteration.

We use the following uniform Aleksandrov--Jerison maximum principle from \cite{Le_Var}.
\begin{thm}[Uniform Aleksandrov--Jerison maximum principle]\label{AJ_thm}\cite[Theorem 1.4]{Le_Var} 
  Let  $\Omega\subset\R^n$ be a bounded convex domain. Let $u,\tilde u\in C(\overline{\Omega})$ be convex functions
with $u=\tilde u=0$ on $\p \Omega$ and $\mu_u\geq\mu_{\tilde u}$ in $\Omega$. Then for all $\theta\in [0, 1]$ and all $x_0\in\Omega$, we have
\[
|\tilde u(x_0)-u(x_0)|^{n}\leq C(n)[\diam (\Omega)]^{n-1}\dist^{\theta}(x_0,\partial \Omega)\int_{\Omega}\dist^{1-\theta}(\cdot,\partial \Omega)\,(d\mu_u-d\mu_{\tilde u}).
\]
\end{thm}

The maximum principle gives the following scale-invariant bound. When $p\neq n$, it controls the $L^\infty$ norm after normalization.

\begin{prop}\label{unibd}
Let $\Omega \subset \R^n$ be a bounded convex domain. Assume $0 \leq \alpha < 2$, $p > 0$, and $M>0$.
Let $u\in C(\overline{\Omega})$ be a nonzero convex solution to
\[
\mu_u=M|u|^p \dist(\cdot,\partial\Omega)^{-\alpha}\,d\mathcal{L}^n\quad\text{in }\Omega,
\qquad u=0\quad\text{on }\partial\Omega.
\]
Then,
\[
0 < c(n, \Omega, \alpha) \leq M\|u\|_{L^\infty(\Omega)}^{p-n} \leq C(n, \Omega, \alpha, p).
\]
\end{prop}

\begin{proof} The argument is adapted from the proof of \cite[Theorem 6.4]{Le_Var}.       
We include the details for the reader's convenience.
As in the proof of \cite[Proposition 6.36]{Le_24}, we have $u \in C^2(\Omega)$.  
Taking $\tilde{u} = 0$ in Theorem \ref{AJ_thm}, we have for all $\theta\in[0, 1]$
\[
\begin{aligned}
|u(x)|^n &\leq C(\Omega, n) \dist(x,\partial\Omega)^\theta \int_{\Omega} \dist(\cdot,\partial\Omega)^{1-\theta}\,d\mu_u \\
&= C M \dist(x,\partial\Omega)^\theta \int_{\Omega} \dist(\cdot,\partial\Omega)^{1-\theta - \alpha} |u|^p\,dx, \quad \text{for } x\in \Omega.
\end{aligned}
\]
Suppose $u$ achieves its minimum at $x_0 \in \Omega$. Then, by choosing $\theta = 0$, we have
\[
\|u\|_{L^\infty(\Omega)}^n = |u(x_0)|^n \leq C M \|u\|_{L^\infty(\Omega)}^p \int_\Omega \dist(\cdot,\partial\Omega)^{1-\alpha}\,dx \leq C(n, \Omega, \alpha) M \|u\|_{L^\infty(\Omega)}^p.
\]
Hence, 
\begin{equation}\label{bd1}
M\|u\|_{L^\infty(\Omega)}^{p - n} \geq c(n, \Omega, \alpha).
\end{equation}

On the other hand, let $w := u/\|u\|_{L^\infty(\Omega)}$. Then, $w$ is a convex function with 
\[
\|w\|_{L^\infty(\Omega)} = 1 \quad \text{and} \quad \mu_w = M\|u\|_{L^\infty(\Omega)}^{p-n}|w|^p\dist(\cdot,\partial\Omega)^{-\alpha} \; d\mathcal{L}^n.
\]
Fix $r > 0$ small and let $\Omega^r = \{x \in \Omega : \dist(x,\partial\Omega) > r\}$. Then by convexity of $w$, 
\[
|Dw| \leq \frac{\|w\|_{L^\infty(\Omega)}}{r} \quad \text{and} \quad |w|\geq \frac{\dist(\cdot,\partial\Omega)}{\diam(\Omega)}\|w\|_{L^\infty(\Omega)} \geq C_1 \quad \text{in } \Omega^r.
\]
Thus,
\begin{equation}\label{bd2}
\begin{aligned}
 M\|u\|_{L^\infty(\Omega)}^{p-n}
 &= \frac{\mu_w(\Omega^r)}{\int_{\Omega^r} |w|^p\dist(\cdot,\partial\Omega)^{-\alpha}\,dx} \\
 &\leq \frac{|B_{1/r}(0)|}{C_1^p \int_{\Omega^r} \dist(\cdot,\partial\Omega)^{-\alpha}\,dx} \\
 &= \frac{C(n, \Omega) r^{-n}}{C_1^p\int_{\Omega^r} \dist(\cdot,\partial\Omega)^{ - \alpha}\,dx} \leq C(n, \Omega, \alpha, p).
\end{aligned}
\end{equation}
Therefore, combining \eqref{bd1} and \eqref{bd2}, we have
\[
c(n, \Omega, \alpha) \leq M\|u\|_{L^\infty(\Omega)}^{p-n} \leq C(n, \Omega, \alpha, p).
\]
\end{proof}

\medskip
We now classify the global boundary regularity of solutions to \eqref{main_eq} according to the value of $p-\alpha$. This proves parts~(i) and~(iii) of Theorem~\ref{thm:intro-Lip-alpha}.
\begin{thm}\label{thm: Lip_alpha}
Let $\Omega \subset \R^n (n \geq 2)$ be a bounded convex domain, $0 \leq \alpha < 2$, $p>0$, and $M>0$. Let $u \in C(\overline{\Omega})$ be a nonzero convex solution to
\[
\mu_u=M|u|^p \dist(\cdot,\partial\Omega)^{-\alpha}\,d\mathcal{L}^n\quad\text{in }\Omega,
\qquad u=0\quad\text{on }\partial\Omega.
\]
\begin{itemize}
        \item[(i)] If $p-\alpha> n - 2$, then $u$ is globally Lipschitz with the estimate
                \[
                |u(x)| \leq C(n, \Omega, \alpha, p) \dist(x,\partial\Omega) \|u\|_{L^\infty(\Omega)}.
                \]
        \item[(ii)] If $p-\alpha \leq n - 2$, then for every $\beta \in (0, \frac{2-\alpha}{n-p})$, we have
                \begin{equation}\label{u_holder}
                |u(x)| \leq C(n, \Omega, \alpha, M, p, \beta) \dist(x,\partial\Omega)^\beta.
                \end{equation}

\end{itemize}
\end{thm}

\begin{proof}
When $\alpha=0$, the conclusions follow from \cite[Proposition 1]{Le_22} and \cite[Theorems 1.1(i) and 1.5(i)]{Le_EVP}, after rescaling in the cases $p\neq n$. Hence, for the remainder of the proof, we assume $0<\alpha<2$.

Let $K=\|u\|_{L^\infty(\Omega)}$ and $v=u/K$. Then $\|v\|_{L^\infty(\Omega)}=1$ and by Proposition \ref{unibd},
\[
\mu_v =A |v|^p \dist(\cdot,\partial\Omega)^{-\alpha} \; d \mathcal{L}^n, \qquad A:=M K^{p-n} \leq C(n, \Omega, \alpha, p).
\]

The proof is based on an iteration argument; see
\cite[proof of Proposition 5.3]{Le_18} and
\cite[proof of Proposition 1]{Le_22} for similar arguments.

\medskip
{\bf Step 1: Base case.}
We first obtain the starting H\"older estimate from Theorem \ref{AJ_thm}. Set
\[
\theta:=\frac12\min\{1,2-\alpha\}
\quad \text{and} \quad
\beta_0:=\frac{\theta}{n}.
\]
By Theorem \ref{AJ_thm}, we have, for every $x\in\Omega$,
\[
\begin{aligned}
|v(x)|^n
&\leq C(n,\Omega)\dist(x,\partial\Omega)^\theta\int_\Omega \dist(\cdot,\partial\Omega)^{1-\theta}\,d\mu_v  \\
&= C(n,\Omega)A \dist(x,\partial\Omega)^\theta\int_\Omega |v|^p \dist(\cdot,\partial\Omega)^{1-\theta-\alpha}\,dx  \\
&\leq C(n,\Omega,\alpha,p)\dist(x,\partial\Omega)^\theta,
\end{aligned}
\]
because $1-\theta-\alpha>-1$ and hence $\dist(\cdot,\partial\Omega)^{1-\theta-\alpha}\in L^1(\Omega)$. Therefore
\[
|v(x)|\leq C \dist(x,\partial\Omega)^{\beta_0}, \qquad x\in\Omega.
\]

\medskip
{\bf Step 2: Iteration.} We now iterate this estimate. Suppose that, for some $\beta_k\in(0,1)$ and $k \in \mathbb{N}$,
\begin{equation}\label{holder_step}
|v(x)|\leq C_k \dist(x,\partial\Omega)^{\beta_k}, \qquad x\in\Omega.
\end{equation}
Since $v<0$ in $\Omega$, \eqref{holder_step} implies
\[
\dist(x,\partial\Omega)^{-\alpha}\leq C_k^{\alpha/\beta_k}|v(x)|^{-\alpha/\beta_k}.
\]
Consequently,
\begin{equation}\label{eq:qk-upper}
\mu_v\leq A_k |v|^{q_k}\,d\mathcal{L}^n,
\qquad
q_k:=p-\frac{\alpha}{\beta_k},
\qquad
A_k:=AC_k^{\alpha/\beta_k}.
\end{equation}

If $q_k>n-2$, then 
$n - 2 < q_k=p-\alpha/\beta_k<p-\alpha$. Hence, this case belongs to part (i).
Set $r_k:=\min\{q_k,n-1\}$.
Since $\|v\|_{L^\infty(\Omega)}=1$,
\[
\mu_v\leq A_k|v|^{r_k}\,d\mathcal{L}^n,
\qquad n-2<r_k<n.
\]
Let $w_k\in C(\overline{\Omega})$ be the nonzero convex solution to
\[
\mu_{w_k}=A_k|w_k|^{r_k}\,d\mathcal{L}^n\quad\text{in }\Omega,
\qquad w_k=0\quad\text{on }\partial\Omega,
\]
whose existence follows from \cite[Theorem 6.4]{Le_Var}. By \cite[Theorem 1.1(i)]{Le_EVP},
\[
|w_k(x)|\leq C \dist(x,\partial\Omega),\qquad x\in\Omega.
\]
Theorem~\ref{thm:subcritical-comparison}, with $\nu=A_k\,d\mathcal{L}^n$, gives $w_k\leq v$ in $\Omega$. Hence
\[
|v(x)|\leq |w_k(x)|\leq C \dist(x,\partial\Omega),\qquad x\in\Omega,
\]
which proves part~(i).

Suppose next that $q_k\leq n-2$, and choose
\[
0<\beta_{k+1}<\frac{2}{n-q_k}
=\frac{2}{n-p+\alpha/\beta_k}.
\]
Let $w_k\in C(\overline{\Omega})$ be a nonzero convex solution to
\[
\mu_{w_k}=A_k|w_k|^{q_k}\,d\mathcal{L}^n\quad\text{in }\Omega,
\qquad w_k=0\quad\text{on }\partial\Omega.
\]
Such a solution exists, after rescaling when necessary, by \cite[Theorem 1.1(i)]{Le_22} when $q_k<0$, by \cite[Theorem 2.14]{Le_Var} when $q_k=0$, and by \cite[Theorem 6.4]{Le_Var} when $0<q_k\leq n-2$. The estimates in \cite[Theorem 1.1(ii)]{Le_22}, \cite[Lemma 1]{Caffarelli}, \cite[Proposition 1]{Le_22}, and \cite[Theorem 1.5(i)]{Le_EVP}, respectively for
\[
q_k<0,\qquad q_k=0<n-2,\qquad 0<q_k<n-2,
\qquad q_k=n-2,
\]
give
\[
|w_k(x)|\leq C_{k+1}\dist(x,\partial\Omega)^{\beta_{k+1}},
\qquad x\in\Omega.
\]
Notice that when $q_k=n-2$, we have $2/(n-q_k)=1$. The log-Lipschitz estimate in \cite[Theorem 1.5(i)]{Le_EVP} gives the above H\"older estimate for every $\beta_{k+1}\in(0,1)$.
By the comparison principles (see Remark~\ref{rem:qk-comparison}), $w_k\leq v$ in $\Omega$. Therefore,
\[
|v(x)|\leq C_{k+1}\dist(x,\partial\Omega)^{\beta_{k+1}},
\qquad x\in\Omega.
\]
Let $a_k = 1/\beta_k$. Then, we can choose $a_{k+1}$ such that
\[
 \frac{n - p}{2} + \frac{\alpha}{2}a_k < a_{k+1} < \frac{n - p}{2} + \frac{\alpha}{2}a_k + \varepsilon,
\]
where $\varepsilon > 0$ will be chosen later.
The iteration stops once $q_k>n-2$; until then, the sequence $\{a_k\}$ satisfies
\[
L_k:= \left(\frac{\alpha}{2}\right)^k a_0 + \frac{n - p}{2} \sum_{i=0}^{k-1} \left(\frac{\alpha}{2}\right)^i < a_{k} < \left(\frac{\alpha}{2}\right)^k a_0 + \frac{n - p}{2} \sum_{i=0}^{k-1} \left(\frac{\alpha}{2}\right)^i + \varepsilon \sum_{i=0}^{k-1} \left(\frac{\alpha}{2}\right)^i =: R_k.
\]
Notice that 
\[
L_k \longrightarrow \frac{n-p}{2-\alpha} \quad \text{and} \quad R_k \longrightarrow \frac{n-p}{2-\alpha} + \frac{\varepsilon}{1-\alpha/2} \quad \text{as } k \longrightarrow \infty.
\]

\medskip
{\bf Case 1: $p - \alpha > n - 2$.} We show that there exists $N\in\mathbb N$ such that
$q_N>n-2$, or equivalently $a_N<\frac{p-n+2}{\alpha}=:A_*$.
Indeed, in this case, we have \[
\frac{n-p}{2-\alpha} < \frac{p-n+2}{\alpha}.
\]
Hence, we can choose $\varepsilon > 0$ small such that
\[
\frac{n-p}{2-\alpha} + \frac{\varepsilon}{1-\alpha/2} < \frac{p-n+2}{\alpha}.
\]
Therefore, we can find $N\in\mathbb N$ such that $p-\frac{\alpha}{\beta_N}>n-2$.
Applying \cite[Theorem 1.1]{Le_EVP} yields $|v(x)|\leq C \dist(x,\partial\Omega)$ in $\Omega$. The proof of (i) is complete.

{\bf Case 2: $p - \alpha \leq n - 2$.} Fix $\beta \in (0, \frac{2-\alpha}{n-p})$.
We can choose $\varepsilon > 0$ small such that
\[
\frac{n-p}{2-\alpha} + \frac{\varepsilon}{1-\alpha/2} < \frac{1}{\beta}.
\]
Hence, we can find $N \in \mathbb N$ such that $a_N < 1/\beta$, or equivalently $\beta_N > \beta$, which gives \eqref{u_holder}. 

\end{proof}

\begin{rem}\label{rem:qk-comparison}
The exponent $q_k$ in the iteration above may be positive, zero, or negative. To compare with the corresponding solution in the Lebesgue case, we use Theorem~\ref{thm:subcritical-comparison} when $0<q_k<n$, the standard comparison principle \cite[Theorem 3.21]{Le_24} when $q_k=0$, and Lemma~\ref{lem:negative-power-comparison} when $q_k<0$. If $q_k\geq n$, the normalization $\|v\|_{L^\infty(\Omega)}=1$ allows us to reduce the exponent to $n-1$.
\end{rem}

\begin{lem}[Comparison principle for negative powers]\label{lem:negative-power-comparison}
Let $\Omega\subset\R^n$ be a bounded convex domain, let $q<0$, and let $A>0$. Let $u,w\in C(\overline\Omega)$ be convex functions such that
\[
w\leq u\leq0\quad\text{on }\partial\Omega,
\qquad u,w<0\quad\text{in }\Omega.
\]
If
\[
\mu_u\leq A|u|^q\,d\mathcal{L}^n,
\qquad
\mu_w\geq A|w|^q\,d\mathcal{L}^n
\quad\text{in }\Omega,
\]
then $w\leq u$ in $\Omega$.
\end{lem}

\begin{proof}
Suppose that $w>u$ somewhere in $\Omega$. For
\[
0<\varepsilon<\sup_\Omega(w-u),
\]
set
\[
D:=\{x\in\Omega:w(x)-u(x)>\varepsilon\},
\qquad \widetilde w:=w-\varepsilon.
\]
Then $D\Subset\Omega$ is nonempty, $\widetilde w=u$ on $\partial D$, and
$0>w>\widetilde w>u$ in $D$. The standard maximum principle
\cite[Lemma 3.11]{Le_24} gives
\[
\mu_u(D)\geq\mu_{\widetilde w}(D)=\mu_w(D).
\]
On the other hand, since $q<0$,
\[
\mu_u(D)\leq A\int_D|u|^q\,dx
<A\int_D|w|^q\,dx
\leq\mu_w(D),
\]
a contradiction.
\end{proof}

\subsection{Flat boundary estimates}\label{sec:flat-boundary-estimates}
We now prove Theorem \ref{thm: ubd_grad}. We first prove a local estimate under the following geometric conditions.

\medskip
{\bf Local geometry.} Assume that, for some $R,h>0$,
\begin{equation}\label{con:geom}
    \begin{aligned}
 B'_{R} \times (0, 2h) := \{(x', x_n) \in \; &\R^{n-1} \times \R : |x'| < R, 0 < x_n < 2h\} \subset \Omega,\\
\{(x',0): |x'|\leq 4R\}\subset\Gamma\subset\partial\Omega, \quad &\text{and }
    \dist(x,\partial\Omega)=t \quad\text{for }x=(x',t)\in B'_{R} \times (0, 2h).
    \end{aligned}
\end{equation}

Define 
\[
H(t):=-u(0,t),
\qquad 0\leq t<2h.
\]
Since $u$ is convex, $H$ is concave. Moreover,
\[
H(0)=0,
\qquad
H(t)>0\quad(0<t<2h),
\]
and $H$ is differentiable almost
everywhere on $(0,2h)$.

\medskip
The following lemma gives an ODE-type inequality for $H$. We obtain it by estimating the Monge--Amp\`{e}re measure of $u$ from above and below in a small cone near the boundary.

\begin{lem}[Cone mass estimate]\label{lem:cone-mass}
Let $\Omega\subset\mathbb{R}^n (n\geq 2)$ be a bounded convex domain. Assume \eqref{con:geom} holds for some $R, h > 0$, and $u\in C(\overline{\Omega})$ is convex in $\Omega$ and $u = 0$ on $\partial \Omega$.
Let
$0\leq \alpha < 2$ and $p > 0$, and suppose
\begin{equation}\label{eq:local-measure}
\mu_u\geq |u|^p \dist(\cdot,\partial\Omega)^{-\alpha}\,d\mathcal{L}^n \qquad \text{in } B'_R \times (0, 2h).
\end{equation}
Define 
\[H(t):=-u(0,t), \qquad 0\leq t<2h.
\] Then
\begin{equation}\label{eq:profile-inequality}
\bigl(H(t)-tH'(t)\bigr)^n
\geq c(n,p,\alpha)R^{2(n-1)}
 t^{2-\alpha}H(t)^p
\qquad\text{for almost every }t\in(0,h).
\end{equation}
\end{lem}

\medskip
\begin{proof}
Fix $t \in (0, h)$ such that $H$ is differentiable at $t$ and set $z_t:=(0,t)$.

Choose any $p_t=(p_t',p_{t,n})\in\partial u(z_t)$. Notice that $p_{t,n} = -H'(t)$. Indeed, by restricting to the $x_n$-axis, the definition of the subdifferential gives
\[
g(\tau) := u(0, \tau) \geq u(0, t) + p_{t,n}(\tau - t), \quad \text{for }0<\tau<2h. 
\]
Hence $p_{t,n}\in\partial g(t)$. Since $g$ is differentiable
at $t$, $\{g'(t)\} = \partial g(t)$, and thus 
\begin{equation}\label{eq:vertical-component}
p_{t,n}=g'(t)=-H'(t).
\end{equation}

Define 
\[
w_t(x):=u(x)-u(z_t)-p_t\cdot(x-z_t).
\]
Since $u$ is convex, $w_t$ is also convex. Moreover, 
\begin{equation}\label{eq:w-basic}
w_t\geq0\quad\text{in }\Omega,
\qquad
w_t(z_t)=0.
\end{equation}
By continuity, $w_t\geq0$ also on $\overline\Omega$. Set
\[
A(t):=H(t)-tH'(t).
\]
Notice that concavity of $H$ and $H(0)=0$ imply $A(t)\geq0$. For $|y'|\leq R$,
recalling \eqref{eq:vertical-component} and $u(y',0)=0$, we have
\begin{equation}\label{eq:w-base}
w_t(y',0)=A(t)-p_t'\cdot y'\qquad\text{for }|y'|\leq R.
\end{equation}
Since $w_t(y',0)\geq0$, by choosing $y'=Rp_t'/|p_t'|$ if $p_t'\neq 0$, we obtain for all $|y'|\leq R$,
\[
|p_t'|\leq\frac{A(t)}R\qquad\text{at the fixed }t\in(0,h).
\]
It follows from \eqref{eq:w-base} that
\begin{equation}\label{eq:w-base-osc}
0\leq w_t(y',0)\leq2A(t) \quad \text{for } |y'|\leq R.
\end{equation}

Let
\[
K_t:=\left\{
\bigl((1-\theta)y',\theta t\bigr):
|y'|<R,\ 0<\theta\leq 1
\right\} \subset \Omega
\] be the cone with vertex at $z_t$ and base $B'_R\times\{0\}$.
Using convexity of $w_t$, \eqref{eq:w-basic}, and
\eqref{eq:w-base-osc}, we obtain
\begin{equation}\label{eq:w-cone}
0\leq w_t\leq2A(t)
\qquad\text{in }K_t.
\end{equation}
The horizontal section of $K_t$ at height $\tau\in(0,t)$ is the disk
$B'_{(1-\tau/t)R}\times\{\tau\}$.

Consider the cylinder
\[
E_t:=B'_{R/8}\times
\left(\frac{7t}{16},\frac{9t}{16}\right) \subset K_t.
\]
Also, notice that by direct computation, for $x\in E_t$,
\[
x\pm\frac{R}{16}e_i\in K_t
\quad(i=1,\dots,n-1),
\qquad
x\pm\frac{t}{16}e_n\in K_t.
\]

Fix $x\in E_t$ and $q\in\partial w_t(x)$. By the definition of the subdifferential, for each $i=1,\dots,n-1$ and $y\in K_t$,
\[
w_t(y)\geq w_t(x)+q\cdot(y-x).
\]
Using \eqref{eq:w-cone}, we have
\[
q\cdot(y-x) \leq 2A(t) \quad\text{for } y\in K_t.
\]
Taking $y = x \pm \frac{R}{16}e_i$ for $i=1,\dots,n-1$, and $y = x \pm \frac{t}{16}e_n$, we obtain
\[
|q_i|\leq32\frac{A(t)}R
\quad(i=1,\dots,n-1),
\qquad
|q_n|\leq32\frac{A(t)}t\qquad\text{for }x\in E_t,\ q\in\partial w_t(x).
\]
Consequently,
\begin{equation}\label{eq:mass-upper}
\mu_u(E_t) = \mu_{w_t}(E_t)=|\partial w_t(E_t)|
\leq C(n)\frac{A(t)^n}{R^{n-1}t}.
\end{equation}

We now estimate the same mass from below. Let $x=(x',\tau)\in E_t$.  
Let $y$ be the intersection of the ray from $z_t$ to $x$ with the base of the cone on $\partial \Omega$. 
Then, by convexity of $u$ and recalling that $u = 0$ on $\partial \Omega$, we have 
\[
u(x) \leq \frac{\tau}{t} u(z_t) + (1-\frac{\tau}{t})u(y) = \frac{\tau}{t} u(z_t) = -\frac{\tau}{t}H(t).
\]
Since $7/16 \leq \tau/t\leq 9/16$ and $u \leq 0$ in $\Omega$, we have
\begin{equation}\label{eq:absu-lower-on-E}
|u(x)|\geq \frac{7}{16}H(t) \quad \text{for } x\in E_t.
\end{equation}
Using \eqref{eq:local-measure} and \eqref{eq:absu-lower-on-E}, we obtain
\begin{equation}\label{eq:mass-lower}
\begin{aligned}
\mu_u(E_t) &\geq \int_{E_t} |u|^p \dist(\cdot,\partial\Omega)^{-\alpha}\,dx\\
            &\geq |E_t| \left( \frac{7}{16}\right)^p H(t)^p \left( \frac{9t}{16}\right)^{-\alpha} \\
            &\geq c(n, p, \alpha) R^{n-1} t^{1-\alpha} H(t)^p.
\end{aligned}
\end{equation}
Combining \eqref{eq:mass-upper} and \eqref{eq:mass-lower} shows that
\[
c(n, p, \alpha) R^{n-1} t^{1-\alpha} H(t)^p \leq C(n)\frac{A(t)^n}{R^{n-1}t}.
\]
Hence, 
\[
c(n, p, \alpha) R^{2(n-1)} t^{2-\alpha} H(t)^p \leq A(t)^n,
\]
proving \eqref{eq:profile-inequality}.
\end{proof}

\medskip
Finally, we prove Theorem \ref{thm: ubd_grad}.
\begin{proof}[Proof of Theorem \ref{thm: ubd_grad}]
Fix $z=(z',z_n)\in\Omega$ sufficiently close to $\Gamma'$. 
By translating and rotating coordinates, we
may assume that the local geometry \eqref{con:geom} holds for some $R,h>0$, depending only on $n, \Omega$, $\Gamma'$ and $\Gamma$,
and that $z = (0,z_n)\in B'_R\times(0,2h)$.

Define
\[
Q(t):=\frac{H(t)}t.
\]
Since $H$ is concave and $H(0)=0$, $Q$ is nonincreasing. Moreover, for almost every $t \in (0, h)$,
\begin{equation}\label{eq:A-Q}
H(t)-tH'(t)=-t^2Q'(t).
\end{equation}
Set
\[
\delta:=n-2-(p-\alpha)\geq0,
\qquad
r:=1-\frac pn=\frac{n-p}{n} > 0.
\]

From Lemma \ref{lem:cone-mass} and \eqref{eq:A-Q}, we have, for almost every $t\in(0,h)$,
\[
    \begin{aligned}
        -Q'(t) &\geq c_0 t^{\frac{p+2-\alpha}{n}-2}Q(t)^{p/n}\\
                &=c_0t^{-1-\delta/n}Q(t)^{p/n},
    \end{aligned}
\]
where
\[
c_0=c_0(n,p,\alpha)R^{2(n-1)/n}>0.
\]
Because $H$ is concave on $(0,2h)$, it is locally Lipschitz there.
Consequently, $Q=H/t$ is locally absolutely continuous on $(0,2h)$.
Moreover, $Q>0$ on $(0,2h)$, so $Q^r$ is absolutely continuous on every
compact interval contained in $(0,2h)$. The chain rule therefore gives, for
almost every $t\in(0,h)$,
\begin{equation}\label{eq:unified-chain}
-\frac{d}{dt}Q(t)^r
=rQ(t)^{-p/n}\bigl(-Q'(t)\bigr)
\geq c_1t^{-1-\delta/n},\qquad c_1=c_1(n,p,\alpha,R)>0.
\end{equation}
In particular, this almost-everywhere differential inequality may be
integrated over $[t,h]$ for every $0<t<h$.

\medskip
We first show \eqref{eq:log-conclusion}. Suppose that $p-\alpha=n-2$, so that
$\delta=0$. For $0<t<h$, integrating \eqref{eq:unified-chain} from $t$ to
$h$ gives
\[
    \begin{aligned}
        Q(t)^r-Q(h)^r
        &=\int_t^h-\frac{d}{ds}Q(s)^r\,ds
        \geq c_1\int_t^h s^{-1}\,ds
        =c_1\log\frac{h}{t}.
    \end{aligned}
\]
Hence,
\[
Q(t)^r\geq Q(h)^r+c_1\log\frac{h}{t}.
\]
Since $r=\frac{n-p}{n}=\frac{2-\alpha}{n}$ and $z = (0,z_n)\in B'_R\times(0,2h)$, 
where $z_n >0$ is small (due to $z$ being close to $\Gamma'$),
\[
\begin{aligned}
|u(z)|=H(z_n)=z_n Q(z_n)
&\geq c\,z_n
\left(\log\frac{h}{z_n}\right)^{\frac{n}{2-\alpha}}\\
&\geq c(n,p,\alpha,\Omega, \Gamma, \Gamma')\,z_n|\log z_n|^{\frac{n}{2-\alpha}}.
\end{aligned}
\]

\medskip
Next, we show \eqref{eq:power-conclusion}. Suppose that
$p-\alpha<n-2$. For $0<t<h$, integrating \eqref{eq:unified-chain} from $t$
to $h$ again gives
\[
Q(t)^r
\geq Q(h)^r+
\frac{nc_1}{\delta}
\left(t^{-\delta/n}-h^{-\delta/n}\right).
\]
Therefore, using $z = (0,z_n)\in B'_R\times(0,2h)$, 
where $z_n >0$ is small,
\[
Q(z_n)^r\geq c(n,p,\alpha,\Omega, \Gamma, \Gamma')\,z_n^{-\delta/n}.
\]
Hence,
\[
\begin{aligned}
|u(z)|=H(z_n)=z_n Q(z_n)
&\geq c\,z_n^{1-\frac{\delta}{nr}}\\
&=c(n,p,\alpha,\Omega, \Gamma, \Gamma')\,z_n^{\frac{2-\alpha}{n-p}}.
\end{aligned}
\]
\end{proof}


\medskip
\subsection{Energy estimates}\label{sec:energy-estimates}
Recall that the Monge--Amp\`ere energy of a convex function $u\in C(\overline{\Omega})$ with zero boundary values is defined by
\[
I[u] = I[u; \Omega] := \int_\Omega |u| \; d\mu_u.
\]
For a solution to \eqref{main_eq}, this is $M\int_\Omega |u|^{p+1}\dist(\cdot,\partial\Omega)^{-\alpha}\,dx$. The upper estimates in Theorem~\ref{thm: Lip_alpha} give a sufficient condition for finite energy on every bounded convex domain. The lower estimates in Theorem~\ref{thm:intro-inf-bd} give necessity when the boundary contains a flat part.

\begin{proof}[Proof of Theorem \ref{thm:intro-energy-estimates}]
  We prove (i) first. If $p-\alpha\leq n-2$, choose
  $\max\{0,(\alpha-1)/(p+1)\}<\beta<(2-\alpha)/(n-p)$;
  otherwise, set $\beta=1$.
  By Theorem \ref{thm: Lip_alpha}, $|u(x)|\leq C \dist(x,\partial\Omega)^\beta$, where $C$ may depend on $\|u\|_{L^\infty(\Omega)}$.
  Since $\beta(p+1)-\alpha>-1$, we obtain
  \[
  \begin{aligned}
    I[u] = \int_\Omega |u| \; d\mu_u &= M\int_\Omega |u|^{p+1}\dist(\cdot,\partial\Omega)^{-\alpha}\,dx\\
    & \leq MC^{p+1}\int_\Omega \dist(\cdot,\partial\Omega)^{\beta(p+1)-\alpha}\,dx<\infty.
  \end{aligned}
  \]

  \medskip
  Next, we show (ii). By translating and rotating coordinates, 
  we may assume that the local geometry \eqref{con:geom} holds for some $R,h>0$, depending only on $n, \Omega$ and $\Gamma$.
  Then, by Theorem \ref{thm:intro-inf-bd}, for $x \in B'_{R_1}\times (0,h_1)$, where $0 < R_1 < R, 0 < h_1 < h$ depend only on $n, \Omega$ and $\Gamma$, 
  we have
  \[
  |u(x)| \geq c \dist(x,\partial\Omega)^{\frac{2-\alpha}{n - p}} = c(n, \Omega, \alpha, \Gamma, p)x_n^{\frac{2-\alpha}{n - p}}.
  \]
  Now, 
  \[
  \begin{aligned}
    I[u] = \int_\Omega |u| \; d\mu_u &= M\int_\Omega |u|^{p+1}\dist(\cdot,\partial\Omega)^{-\alpha}\,dx\\
    & \geq c^{p+1}M \int_{B'_{R_1}\times (0,h_1)} \dist(\cdot,\partial\Omega)^{\frac{2-\alpha}{n-p} (p+1) - \alpha}\,dx\\
    & = c^{p+1}M \int_{B'_{R_1}} \int_0^{h_1} t^{\frac{2-\alpha}{n-p} (p+1) - \alpha} \; dt dx' = \infty.
  \end{aligned}
  \]
  The last integral diverges since 
\[
\frac{(2-\alpha)(p+1)}{n-p} - \alpha \leq -1,
\] 
due to the assumption $p - \alpha \leq (\alpha - 1)n - 2$.
  The proof is complete.
\end{proof}


\medskip
\section{Endpoint upper estimates}
\label{sec:refined-estimates}
In this section, we prove the log-Lipschitz estimate in Theorem~\ref{thm:intro-Lip-alpha}(ii)
and the endpoint H\"older estimate in part~(iv).
We then combine Theorem~\ref{thm: refined_holder} with Theorem~\ref{thm: Lip_alpha} to complete the proof of Theorem~\ref{thm:intro-Lip-alpha}.
\begin{thm}[Endpoint estimates]\label{thm: refined_holder}
Let $\Omega\subset\R^n (n\geq 2)$ be a bounded convex domain. 
Let $0\leq\alpha<2$ and $p > 0$. Suppose $u\in C(\overline\Omega)$ is a nonzero convex solution to
\[
        \mu_u=|u|^p \dist(\cdot,\partial\Omega)^{-\alpha}\,d\mathcal{L}^n
        \quad\text{in }\Omega,\qquad
        u=0\quad\text{on }\partial\Omega.
\]
\begin{itemize}
    \item[(i)]If $n(\alpha - 1) - 2 < p - \alpha < n-2$, then 
        \begin{equation}\label{eq: refined_subcrit}
            |u(x)|\leq C(n, \Omega, \alpha, p) \dist(x,\partial\Omega)^{\frac{2-\alpha}{n-p}} \qquad \text{for all }x\in\Omega;
        \end{equation}

    \item[(ii)] If $ p - \alpha = n-2$, then 
        \begin{equation}\label{eq: refined_crit}
        |u(x)| \leq  C(n, \Omega, \alpha, p) \dist(x,\partial\Omega)\left(1 + |\log \dist(x,\partial\Omega)|^{\frac{n}{2-\alpha}}\right)\qquad\text{for all }x\in\Omega.
        \end{equation}
\end{itemize}
\end{thm}

The H\"older estimates in Theorem~\ref{thm: Lip_alpha} imply that every nonzero solution in the range of Theorem~\ref{thm: refined_holder} has finite energy; see Theorem~\ref{thm:intro-energy-estimates}(i). Theorem~\ref{thm: p<n_uniq} gives the uniqueness of such solutions, and Lemma~\ref{lem: maximal} then yields the following comparison with convex subsolutions. This comparison reduces the endpoint estimates to the construction of suitable subsolutions.

\begin{prop}[Comparison with a convex subsolution]\label{prop: restricted_comp}
Let $\Omega\subset\R^n (n\geq 2)$ be a bounded convex domain and $0\leq\alpha<2$. 
Let $p > 0$ satisfy $n(\alpha - 1) - 2 < p - \alpha \leq n-2$. 
Suppose that $u\in C(\overline\Omega)$ is a nonzero convex solution to
\begin{equation}\label{eq:restricted_comp}
\mu_u = |u|^p \dist(\cdot,\partial\Omega)^{-\alpha} \; d\mathcal{L}^n \quad \text{in } \Omega, \qquad u = 0 \quad \text{on } \partial \Omega,
\end{equation}
and let $\Phi \in C(\overline\Omega)$ be convex with $\Phi<0$ in $\Omega$,
$\Phi\leq0$ on $\partial\Omega$, and, for some $c > 0$,
\[
\mu_\Phi \geq c|\Phi|^p \dist(\cdot,\partial\Omega)^{-\alpha} \; d\mathcal{L}^n \quad \text{in } \Omega.
\]
Then, there exists a constant $C = C(n, p, c) > 0$ such that
\[
|u(x)| \leq C |\Phi(x)| \quad \text{for all } x \in \Omega. 
\]
\end{prop}
\begin{proof}
Let $\nu := \dist(\cdot,\partial\Omega)^{-\alpha} \; d\mathcal{L}^n$. Since $p\leq n-2+\alpha<n$, we have
$p\in(0,n)$. Moreover, $\nu(\Omega) > 0$ and, since $0 \leq \alpha < 2$,
\[
\int_\Omega \dist(\cdot, \partial \Omega) \; d\nu = \int_\Omega \dist(\cdot, \partial \Omega)^{1-\alpha} \; dx < \infty.
\]
Let $a:=c^{-1/(n-p)}$ and $\Psi := a\Phi$. Then, by the homogeneity of
the Monge--Amp\`ere measure,
\[
        \mu_\Psi
        =
        a^n\mu_\Phi
        \geq
        ca^n|\Phi|^p\nu
        =
        ca^{n-p}|\Psi|^p\nu
        =
        |\Psi|^p\nu
        \quad\text{in }\Omega,
\]
and $\Psi\leq0$ on $\partial\Omega$. By Lemma \ref{lem: maximal}, there exists
a nonzero convex function $u^* \in C(\overline{\Omega})$ satisfying
\eqref{eq:restricted_comp} and
\[
\Psi \leq u^* \leq 0 \quad \text{in } \Omega. 
\]
Since $p-\alpha>(\alpha-1)n-2$, Theorem \ref{thm:intro-energy-estimates}(i)
shows that $u,u^*\in\mathbb E(\Omega)$. Hence Theorem
\ref{thm: p<n_uniq} gives $u = u^*$ in $\Omega$. Therefore,
\[
|u| = |u^*|\leq |\Psi| = a|\Phi| \quad \text{in } \Omega.
\]
\end{proof}

\medskip
To apply Proposition~\ref{prop: restricted_comp}, we now construct a convex subsolution with the boundary estimates in Theorem~\ref{thm: refined_holder}. We adapt Le's functions for the H\"older and log-Lipschitz estimates \cite[Proposition 2.2 and Lemma 4.1]{Le_EVP} and take their supremum over supporting hyperplanes of $\Omega$ to obtain zero boundary values.

\begin{lem}\label{lem: refined_subsoln}
Let $\Omega\subset\R^n (n\geq 2)$ be a bounded convex domain.
Let $0\leq\alpha<2$ and $p>0$ satisfy $p-\alpha\leq n-2$.
Then there is a convex function
$\Phi\in C(\overline\Omega)$ such that $\Phi<0$ in $\Omega$,
$\Phi=0$ on $\partial\Omega$, and
\begin{equation}\label{eq:refined_barrier_sub}
        \mu_\Phi
        \geq
        c_*\,|\Phi|^p \dist(\cdot,\partial\Omega)^{-\alpha}\,d\mathcal{L}^n
        \quad\text{in }\Omega;
\end{equation}
and furthermore:
\begin{itemize}
        \item[(i)] if $p-\alpha<n-2$ and
        \[
                \beta:=\frac{2-\alpha}{n-p},
        \]
        then
        \begin{equation}\label{eq:subcritical_barrier_comp}
                c \dist(x,\partial\Omega)^\beta
                \leq
                |\Phi(x)|
                \leq
                C\dist(x,\partial\Omega)^\beta
                \qquad\text{in }\Omega;
        \end{equation}
        \item[(ii)] if $p-\alpha=n-2$, set
        \[
                \gamma:=\frac{n}{2-\alpha},
                \qquad
                L(t):=\log\frac{R}{t},
                \qquad
                R:=e^{2\gamma}\diam(\Omega).
        \]
        Then
        \begin{equation}\label{eq:Phi_borderline_comp}
                c \dist(x,\partial\Omega)L(\dist(x,\partial\Omega))^\gamma
                \leq
                |\Phi(x)|
                \leq
                C \dist(x,\partial\Omega)L(\dist(x,\partial\Omega))^\gamma
                \qquad\text{in }\Omega.
        \end{equation}
\end{itemize}
Here $c,c_*,C>0$ depend only on $n$, $\Omega$, $p$, and $\alpha$.
\end{lem}

\begin{proof}[Proof of Lemma \ref{lem: refined_subsoln}]
Let $D:=\diam(\Omega)$ and let
\[
        \mathcal S
        :=
        \left\{(y,\nu)\in\partial\Omega\times\mathbb S^{n-1}:
        \Omega\subset \{x\in\mathbb R^n:(x-y)\cdot \nu>0\}\right\}.
\]
For each $y\in\partial\Omega$, there is at least one $\nu$ such that
$(y,\nu)\in\mathcal S$; see, for example, \cite[Theorem 2.10]{Le_24}.
Since $\Omega$ is open, the defining condition is equivalent to $(x-y)\cdot\nu\geq0$ for every $x\in\overline\Omega$.
Thus $\mathcal S$ is closed in the compact set $\partial\Omega\times\mathbb S^{n-1}$ and is compact.
For $(y,\nu)\in\mathcal S$ and $x\in\Omega$, define
\[
        \ell_{y,\nu}(x):=(x-y)\cdot\nu,
        \qquad
        \tau_{y,\nu}(x):=x-y-\ell_{y,\nu}(x)\nu.
\]
Then
\[
        \dist(x,\partial\Omega)\leq\ell_{y,\nu}(x)\leq D
        \qquad\text{for all }x\in\Omega,\ (y,\nu)\in\mathcal S.
\]
In both constructions below, $(y,\nu)\mapsto\phi_{y,\nu}(x)$ is continuous on $\mathcal S$ for each fixed $x\in\Omega$.
Hence the supremum defining $\Phi(x)$ is attained.

\medskip
\noindent{\bf The case $p-\alpha<n-2$.}
Assume $p-\alpha<n-2$ and set
\[
        \beta:=\frac{2-\alpha}{n-p}.
\]
Then $0<\beta<1$ and
\[
        n\beta-2=p\beta-\alpha.
\]
Choose $A>0$ such that
\[
        A>D^2,
        \qquad
        (1-\beta)A-(1+\beta)D^2\geq 1.
\]
For $(y,\nu)\in\mathcal S$, define
\[
        \phi_{y,\nu}(x)
        :=
        \ell_{y,\nu}(x)^\beta
        \bigl(|\tau_{y,\nu}(x)|^2-A\bigr),
        \qquad
        \Phi(x):=\sup_{(y,\nu)\in\mathcal S}\phi_{y,\nu}(x).
\]
Fix $(y,\nu)\in\mathcal S$. After a rotation and a translation, assume
$y=0$, $\nu=e_n$, and write $x=(x',t)$. Then
\[
        \phi_{y,\nu}(x)=t^\beta(|x'|^2-A),
\]
and a direct computation gives
\[
        \det D^2\phi_{y,\nu}
        =
        2^{n-1}\beta t^{n\beta-2}
        \Big((1-\beta)A-(1+\beta)|x'|^2\Big).
\]
Since $|x'|\leq D$ in $\Omega$, the choice of $A$ implies
\[
        \det D^2\phi_{y,\nu}(x)
        \geq c\,\ell_{y,\nu}(x)^{p\beta-\alpha}>0
        \qquad\text{in }\Omega.
\]
The first $(n-1)$ leading principal minors are positive in $\Omega$; hence
$\phi_{y,\nu}$ is convex. Moreover,
\[
        -A\ell_{y,\nu}(x)^\beta
        \leq
        \phi_{y,\nu}(x)
        \leq
        -(A-D^2)\ell_{y,\nu}(x)^\beta
        <0.
\]
Since $\Phi$ is the supremum of convex functions, $\Phi$ is convex. The
upper bound below follows from $\dist(x,\partial\Omega)\leq\ell_{y,\nu}(x)$ for all
$(y,\nu)\in\mathcal S$. For the lower bound, choose
$y\in\partial\Omega$ with $|x-y|=\dist(x,\partial\Omega)$ and
$\nu=(x-y)/|x-y|$; then $(y,\nu)\in\mathcal S$,
$\ell_{y,\nu}(x)=\dist(x,\partial\Omega)$, and $\tau_{y,\nu}(x)=0$. Hence
\[
        -A \dist(x,\partial\Omega)^\beta
        \leq
        \Phi(x)
        \leq
        -(A-D^2)\dist(x,\partial\Omega)^\beta
        \qquad\text{in }\Omega.
\]
In particular, $\Phi<0$ in $\Omega$ and $\Phi=0$ on $\partial\Omega$.

Let $x\in\Omega$ be a point where $\Phi$ is twice differentiable and choose
$(y,\nu)\in\mathcal S$ with $\Phi(x)=\phi_{y,\nu}(x)$. Then
$D^2\Phi(x)\geq D^2\phi_{y,\nu}(x)$ and hence
\[
        \det D^2\Phi(x)
        \geq c\,\ell_{y,\nu}(x)^{p\beta-\alpha}.
\]
For this maximizing pair, the preceding bounds imply
\[
        \dist(x,\partial\Omega)\leq\ell_{y,\nu}(x)
        \leq
        \left(\frac{A}{A-D^2}\right)^{1/\beta}\dist(x,\partial\Omega).
\]
Thus, with $q:=p\beta-\alpha$, we have
$\ell_{y,\nu}(x)^q\geq c \dist(x,\partial\Omega)^q$; this remains true when $q<0$ because
$\ell_{y,\nu}(x)$ and $\dist(x,\partial\Omega)$ are comparable. Therefore
\begin{equation}\label{eq:Phi_det_1}
        \det D^2\Phi(x)\geq c\,\dist(x,\partial\Omega)^{p\beta-\alpha}
        \qquad\text{for a.e. }x\in\Omega.
\end{equation}
The absolutely continuous part of $\mu_\Phi$ has density
$\det D^2\Phi$ almost everywhere, and its singular part is nonnegative.
Combining this with \eqref{eq:Phi_det_1} and
\[
        |\Phi(x)|^p \dist(x,\partial\Omega)^{-\alpha}
        \leq C \dist(x,\partial\Omega)^{p\beta-\alpha},
\]
we obtain \eqref{eq:refined_barrier_sub}. This also proves
\eqref{eq:subcritical_barrier_comp}.

\medskip
\noindent{\bf The case $p-\alpha=n-2$.}
Assume $p-\alpha=n-2$. Define
\[
        \gamma:=\frac{n}{2-\alpha},
        \qquad
        \eta:=\gamma-1=\frac{p}{2-\alpha},
        \qquad
        R:=e^{2\gamma}D,
\]
and, for $0<t\leq D$, set
\[
        L(t):=\log\frac{R}{t}.
\]

Choose
\begin{equation}\label{eq:B_choice}
        B:=\frac{8D^2}{\gamma}.
\end{equation}

For $a>0$, define
\[
        g_a(t):=tL(t)^a.
\]
A direct computation gives
\begin{equation}\label{borderline_g'}
        g_a'(t)=L(t)^a-aL(t)^{a-1},
\end{equation}
and
\[
        g_a''(t)
        =
        -\frac{a}{t}L(t)^{a-1}
        +\frac{a(a-1)}{t}L(t)^{a-2}.
\]
Since $L(t)\geq 2\gamma$ for $0<t\leq D$, for
$a\in\{\eta,\gamma\}$, it is easy to verify that
\begin{equation}\label{eq:g_derivative_bounds}
        g_a'(t)\geq \frac12 L(t)^a,
        \qquad
        0 < \frac{a}{2}\frac{L(t)^{a-1}}{t}
        \leq
        -g_a''(t)
        \leq
        2a\frac{L(t)^{a-1}}{t}.
\end{equation}

For $(y,\nu)\in\mathcal S$, define
\[
        \phi_{y,\nu}(x)
        :=
        g_\eta(\ell_{y,\nu}(x))\bigl(|\tau_{y,\nu}(x)|^2-D^2\bigr)
        -B g_\gamma(\ell_{y,\nu}(x)).
\]
Define
\[
        \Phi(x):=\sup_{(y,\nu)\in\mathcal S}\phi_{y,\nu}(x).
\]

\medskip
{\bf Step 1: Properties of $\phi_{y,\nu}$.} Fix $(y,\nu)\in\mathcal S$.
After a rotation and a translation, we may assume $y=0$, $\nu=e_n$, and write
$x=(x',t)\in \Omega$. Then
\[
\ell_{y, \nu}(x) = t \quad \text{and}\quad \tau_{y, \nu}(x) = (x',0),
\]
and
\[
        \phi(x',t):= \phi_{y,\nu}(x) = g_\eta(t)(|x'|^2-D^2)-Bg_\gamma(t).
\]
A direct computation gives 
\[
        D^2\phi=
        \begin{pmatrix}
        2g_\eta & 0 & \cdots & 0 & 2x_1g_\eta'\\
        0 & 2g_\eta & \cdots & 0 & 2x_2g_\eta'\\
        \vdots & \vdots & \ddots & \vdots & \vdots\\
        0 & 0 & \cdots & 2g_\eta & 2x_{n-1}g_\eta'\\
        2x_1g_\eta' & 2x_2g_\eta' & \cdots & 2x_{n-1}g_\eta'
        & (|x'|^2-D^2)g_\eta''-Bg_\gamma''
        \end{pmatrix}.
\]
Hence,
\begin{equation}\label{eq:borderline_det_exact}
        \det D^2\phi = (2g_\eta)^{n-1}\Big[g_\eta''(|x'|^2-D^2)-Bg_\gamma''-\frac{2(g_\eta')^2}{g_\eta}|x'|^2\Big].
\end{equation}

By \eqref{borderline_g'} and \eqref{eq:g_derivative_bounds}, we obtain
\[
        \frac{(g_\eta')^2}{g_\eta}
        \leq
        \frac{L(t)^\eta}{t}
        \quad \text{and}\quad
        -g_\gamma''(t)\geq \frac{\gamma}{2}\frac{L(t)^{\gamma - 1}}{t} =  \frac{\gamma}{2}\frac{L(t)^\eta}{t}.
\]
Since $|x'|\leq D$ for $x\in\Omega$, the choice \eqref{eq:B_choice} of $B$ gives
\[
        -Bg_\gamma''
        -
        \frac{2(g_\eta')^2}{g_\eta}|x'|^2
        \geq
        \left(\frac{B\gamma}{2}-2D^2\right)\frac{L(t)^\eta}{t}
        \geq
        \frac{B\gamma}{4}\frac{L(t)^\eta}{t}.
\]
Together with $D^2-|x'|^2\geq 0$ and recalling \eqref{eq:borderline_det_exact}, this shows that 
\[
\det D^2\phi \geq 2^{n-1}g_\eta^{n-1}\frac{B\gamma}{4}\frac{L(t)^\eta}{t} = c_1(D, \gamma, n) t^{n-2}L(t)^{n\eta} > 0 \quad \text{in }\Omega.
\]
Notice also that the first $n-1$ leading principal minors of $D^2\phi$ are positive in $\Omega$. Hence, $\phi$ is convex in $\Omega$.

Therefore, we have shown that for all $(y,\nu)\in\mathcal S$, $\phi_{y,\nu}$ is convex in $\Omega$
and satisfies
\begin{equation}\label{eq:single_borderline_det}
        \det D^2\phi_{y,\nu}(x)
        \geq
        c_1\ell_{y,\nu}(x)^{\,n-2}
        L(\ell_{y,\nu}(x))^{n\eta}.
\end{equation}
Also, since $|\tau_{y,\nu}(x)|\leq D$ and $g_\eta\leq g_\gamma$ on $(0,D]$, by the definition of $\phi_{y,\nu}$, we have
\begin{equation}\label{eq:single_borderline_bounds}
        -(D^2+B)g_\gamma(\ell_{y,\nu}(x))
        \leq
        \phi_{y,\nu}(x)
        \leq
        -B g_\gamma(\ell_{y,\nu}(x))
        <0
        \qquad\text{in }\Omega.
\end{equation}

\medskip
{\bf Step 2: Properties of $\Phi$.}
We will show
\begin{itemize}
        \item[(a)] $\Phi$ is convex in $\Omega$ and $\Phi\leq0$ in $\Omega$;
        \item[(b)] $|\Phi|$ is comparable to $\dist(\cdot,\partial\Omega)L(\dist(\cdot,\partial\Omega))^\gamma$ in $\Omega$;
        \item[(c)] $\Phi$ satisfies \eqref{eq:refined_barrier_sub}.
\end{itemize}

\medskip
We first show (a) and (b). Fix $x\in\Omega$. Since $\Phi$ is the supremum of convex functions,
$\Phi$ is also convex in $\Omega$. Moreover, since $g_\gamma$ is increasing and
$\dist(x,\partial\Omega)\leq\ell_{y,\nu}(x)$ for all $(y,\nu)\in\mathcal S$, it follows from
\eqref{eq:single_borderline_bounds} that
\[
        \phi_{y,\nu}(x)
        \leq
        -B g_\gamma(\dist(x,\partial\Omega)).
\]
Taking the supremum over $(y,\nu)\in\mathcal S$ gives the upper bound
\begin{equation}\label{eq: BDL_Phi_upper}
        \Phi(x) \leq -B \dist(x,\partial\Omega)L(\dist(x,\partial\Omega))^\gamma.
\end{equation}
On the other hand, we can choose $y\in\partial\Omega$ with $|x-y|=\dist(x,\partial\Omega)$ and
$\nu=(x-y)/|x-y|$; then $(y,\nu)\in\mathcal S$, $\ell_{y,\nu}(x)=\dist(x,\partial\Omega)$,
and $\tau_{y,\nu}(x)=0$. Hence,
\[
\begin{aligned}
        \Phi(x) \geq \phi_{y,\nu}(x) &\geq -D^2g_\eta(\dist(x,\partial\Omega))-Bg_\gamma(\dist(x,\partial\Omega)) \\
        &\geq -(D^2+B)g_\gamma(\dist(x,\partial\Omega)).
\end{aligned}
\]
Together with \eqref{eq: BDL_Phi_upper}, this gives the precise comparison
\begin{equation}\label{eq:Phi_borderline_precise_comp}
        -(D^2+B)g_\gamma(\dist(x,\partial\Omega))
        \leq
        \Phi(x)
        \leq
        -B g_\gamma(\dist(x,\partial\Omega))
        \qquad\text{in }\Omega.
\end{equation}
In particular, this implies \eqref{eq:Phi_borderline_comp}. By continuous extension,
\[
\Phi < 0 \quad \text{in } \Omega, \quad \Phi = 0 \quad \text{on } \partial \Omega. 
\]

\medskip
Now, we prove (c). By Aleksandrov's theorem (see, for instance, \cite[Theorem 2.89]{Le_24}), 
$\Phi$ is twice differentiable a.e. in $\Omega$. Let $x\in\Omega$ be a point where
$\Phi$ is twice differentiable.
Choose a maximizing pair $(y,\nu)\in\mathcal S$, so that
\[
        \Phi(x)=\phi_{y,\nu}(x).
\]
Since $\Phi\geq\phi_{y,\nu}$ in $\Omega$ with equality at $x$, $\Phi-\phi_{y,\nu}$ has a local minimum at $x$.
Hence, in the sense of symmetric matrices,
\[
        D^2\Phi(x)\geq D^2\phi_{y,\nu}(x) > 0.
\]
Then, from \eqref{eq:single_borderline_det},
\[
        \det D^2\Phi(x) \geq c_1\,\ell_{y,\nu}(x)^{\,n-2} L(\ell_{y,\nu}(x))^{n\eta}.
\]

Next, we will compare $\ell_{y,\nu}(x)$ with $\dist(x,\partial\Omega)$ for this pair $(y, \nu)$.
Since $\Phi(x)=\phi_{y,\nu}(x)$, 
combining \eqref{eq:single_borderline_bounds} and \eqref{eq:Phi_borderline_precise_comp} gives
\[
        g_\gamma(\ell_{y,\nu}(x))\leq \frac{D^2+B}{B}g_\gamma(\dist(x,\partial\Omega)).
\]
By \eqref{borderline_g'} and $L(t)\geq2\gamma$, we have
\[
        \frac{g_\gamma'(t)}{g_\gamma(t)}
        =\frac1t\left(1-\frac{\gamma}{L(t)}\right)
        \geq\frac1{2t}.
\]
Integrating from $s$ to $t$, where $0<s\leq t\leq D$, gives
\[
        \frac{g_\gamma(t)}{g_\gamma(s)}\geq\left(\frac ts\right)^{1/2}.
\]
Thus $g_\gamma(t)\leq C_0g_\gamma(s)$ implies $t\leq C_0^2s$. Moreover,
\[
        L(s)=L(t)+\log\frac ts
        \leq L(t)+2\log C_0
        \leq\left(1+\frac{\log C_0}{\gamma}\right)L(t).
\]
Applying these bounds with $s=\dist(x,\partial\Omega)$, $t=\ell_{y,\nu}(x)$, and
$C_0=(D^2+B)/B$ yields
\[
        \det D^2\Phi(x)
        \geq
        c\,\dist(x,\partial\Omega)^{n-2}L(\dist(x,\partial\Omega))^{n\eta}.
\]
Using $p-\alpha=n-2$, this becomes
\[
        \det D^2\Phi(x)
        \geq
        c\,\dist(x,\partial\Omega)^{p-\alpha}L(\dist(x,\partial\Omega))^{p\gamma}.
\]
On the other hand, \eqref{eq:Phi_borderline_precise_comp} gives
\[
        |\Phi(x)|^p \dist(x,\partial\Omega)^{-\alpha}
        \leq
        C\,\dist(x,\partial\Omega)^{p-\alpha}L(\dist(x,\partial\Omega))^{p\gamma}.
\]
Since convex functions are twice differentiable almost everywhere and the
singular part of the Monge--Amp\`ere measure is nonnegative, the preceding
pointwise estimate implies
\[
        \mu_\Phi
        \geq
        c\,|\Phi|^p \dist(\cdot,\partial\Omega)^{-\alpha}\,d\mathcal{L}^n
        \quad\text{in }\Omega.
\]
This is \eqref{eq:refined_barrier_sub} when $p-\alpha=n-2$. The proof is
complete.
\end{proof}

\medskip
We can now prove Theorem \ref{thm: refined_holder}.
\begin{proof}[Proof of Theorem \ref{thm: refined_holder}]
Let $\Phi \in C(\overline{\Omega})$ be the convex function in Lemma \ref{lem: refined_subsoln}. Recall that 
\[
        \mu_\Phi \geq c_*(n, \Omega, p, \alpha)|\Phi|^p \dist(\cdot,\partial\Omega)^{-\alpha}\,d\mathcal{L}^n \quad \text{and}\quad   \mu_u = |u|^p \dist(\cdot,\partial\Omega)^{-\alpha}\,d\mathcal{L}^n   \quad\text{in }\Omega,
\]
and $\Phi < 0$ in $\Omega$, $u = \Phi = 0$ on $\partial \Omega$. Hence, by Proposition \ref{prop: restricted_comp}, we have
\begin{equation}\label{refined_eq1}
    |u(x)| \leq C(n, \Omega, p, \alpha) |\Phi(x)| \quad \text{for } x\in \Omega.
\end{equation}
Recall again from Lemma \ref{lem: refined_subsoln} that $\Phi$ satisfies, for $C = C(n, \Omega, p, \alpha) > 0$, 
\begin{equation}\label{refined_eq2}
|\Phi(x)| \leq
\begin{cases}
    C\dist(x,\partial\Omega)^{\frac{2-\alpha}{n-p}}, &\text{ if }  n(\alpha - 1) - 2 < p-\alpha<n-2, \\
    C \dist(x,\partial\Omega)\Big|\log \frac{R}{\dist(x,\partial\Omega)}\Big|^{\frac{n}{2-\alpha}}, &\text{ if }  p-\alpha = n-2.
\end{cases}
\end{equation}
Note that \eqref{eq: refined_subcrit} follows easily from \eqref{refined_eq1} and \eqref{refined_eq2}. 
Also, since
\begin{equation}\label{refined_eq3}
\begin{aligned}
\Big|\log \frac{R}{\dist(x,\partial\Omega)}\Big|^{\frac{n}{2-\alpha}}  &= |\log R - \log \dist(x,\partial\Omega)|^{\frac{n}{2-\alpha}} \\
&\leq C(n, \Omega, \alpha) (1 + |\log \dist(x,\partial\Omega)|^{\frac{n}{2-\alpha}}),
\end{aligned}
\end{equation}
we obtain \eqref{eq: refined_crit} by combining \eqref{refined_eq1}, \eqref{refined_eq2} and \eqref{refined_eq3}. The proof is complete.
\end{proof}

\medskip
Finally, we give a quick proof of Theorem \ref{thm:intro-Lip-alpha} by combining Theorem \ref{thm: Lip_alpha} and Theorem \ref{thm: refined_holder}.
\begin{proof}[Proof of Theorem \ref{thm:intro-Lip-alpha}]
Parts~(i) and~(iii) follow from Theorem~\ref{thm: Lip_alpha}(i) and~(ii), respectively.
It remains to prove parts~(ii) and~(iv), for which $n(\alpha-1)-2<p-\alpha\leq n-2$.
Since $\alpha<2$, we have $p<n$. Set $v:=M^{-1/(n-p)}u$. By homogeneity,
\[
\mu_v=|v|^p \dist(\cdot,\partial\Omega)^{-\alpha}\,d\mathcal L^n\quad\text{in }\Omega,
\qquad v=0\quad\text{on }\partial\Omega.
\]
Applying Theorem~\ref{thm: refined_holder}(i) and~(ii) to $v$ and using $u=M^{1/(n-p)}v$, we obtain, for all $x\in\Omega$,
\[
|u(x)|\leq 
\begin{cases}
CM^{1/(n-p)}\dist(x,\partial\Omega)^{\frac{2-\alpha}{n-p}}, &\text{if }p-\alpha<n-2,\\
CM^{1/(n-p)}\dist(x,\partial\Omega)\left(1+|\log \dist(x,\partial\Omega)|^{\frac{n}{2-\alpha}}\right), &\text{if }p-\alpha=n-2,
\end{cases}
\]
where $C=C(n,\Omega,\alpha,p)>0$. These are the estimates in parts~(iv) and~(ii), respectively. The proof is complete.
\end{proof}

\medskip

\enlargethispage{2\baselineskip}


\begin{thebibliography}{999}


\bibitem{ADPZ} Astala, K.; Duse, E.; Prause, I.; Zhong, X. Dimer models and conformal structures.
{\it Comm. Pure Appl. Math.} {\bf 79} (2026), no. 2, 340--446.

\bibitem{Caffarelli} Caffarelli, L. A. A localization property of viscosity solutions to the Monge--Amp\`ere equation and their strict convexity.
{\it Ann. of Math.} (2) {\bf 131} (1990), no. 1, 129--134.

\bibitem{CTW} Caffarelli, L. A.; Tang, L.; Wang, X.-J. Global $C^{1,\alpha}$ regularity for Monge--Amp\`ere equation and convex envelope.
{\it Arch. Ration. Mech. Anal.} {\bf 244} (2022), no. 1, 127--155.

\bibitem{ChengYau} Cheng, S.-Y.; Yau, S.-T. On the regularity of the Monge--Amp\`ere equation $\det(\partial^2u/\partial x_i\partial x_j)=F(x,u)$.
{\it Comm. Pure Appl. Math.} {\bf 30} (1977), no. 1, 41--68.

\bibitem{CHX} Cheng, T.; Huang, G.; Xu, X. Uniqueness of nontrivial solutions for degenerate Monge--Amp\`ere equations.
{\it SIAM J. Math. Anal.} {\bf 56} (2024), no. 1, 234--253.

\bibitem{ChouWang} Chou, K.-S.; Wang, X.-J. The $L_p$-Minkowski problem and the Minkowski problem in centroaffine geometry.
{\it Adv. Math.} {\bf 205} (2006), no. 1, 33--83.

\bibitem{CKP} Cohn, H.; Kenyon, R.; Propp, J. A variational principle for domino tilings.
{\it J. Amer. Math. Soc.} {\bf 14} (2001), no. 2, 297--346.

\bibitem{CollinsFirester} Collins, T. C.; Firester, B. On a general class of free boundary Monge--Amp\`ere equations.
arXiv:2508.05551, preprint.

\bibitem{Donaldson} Donaldson, S. K. Interior estimates for solutions of Abreu's equation.
{\it Collect. Math.} {\bf 56} (2005), no. 2, 103--142.

\bibitem{GT} Gilbarg, D.; Trudinger, N. S. {\em Elliptic partial differential equations of second order.}
Reprint of the 1998 edition. Classics in Mathematics, Springer-Verlag, Berlin, 2001.

\bibitem{HeHuang} He, R.; Huang, G. Weighted eigenvalue problem for a class of singular/degenerate $k$-Hessian equations.
arXiv:2505.03231, {\it Indiana Univ. Math. J.}, to appear.

\bibitem{HHW} Hong, J.; Huang, G.; Wang, W. Existence of global smooth solutions to Dirichlet problem for degenerate elliptic Monge--Amp\`ere equations.
{\it Comm. Partial Differential Equations} {\bf 36} (2011), no. 4, 635--656.

\bibitem{Huang} Huang, G. Uniqueness of least energy solutions for Monge--Amp\`ere functional.
{\it Calc. Var. Partial Differential Equations} {\bf 58} (2019), no. 2, Art. 73, 20 pp.

\bibitem{KO} Kenyon, R.; Okounkov, A. Planar dimers and Harnack curves.
{\it Duke Math. J.} {\bf 131} (2006), no. 3, 499--524.

\bibitem{KOS} Kenyon, R.; Okounkov, A.; Sheffield, S. Dimers and amoebae.
{\it Ann. of Math.} (2) {\bf 163} (2006), no. 3, 1019--1056.

\bibitem{Le_18} Le, N. Q. The eigenvalue problem for the Monge--Amp\`ere operator on general bounded convex domains. 
{\it Ann. Sc. Norm. Super. Pisa Cl. Sci.} (5) {\bf 18} (2018), no. 4, 1519--1559.


\bibitem{Le_22} Le, N. Q. Optimal boundary regularity for some singular Monge--Amp\`ere equations on bounded convex domains.
{\it Discrete Contin. Dyn. Syst.} {\bf 42} (2022), no. 5, 2199--2214.


\bibitem{Le_23} Le, N. Q. Remarks on sharp boundary estimates for singular and degenerate Monge--Amp\`ere equations.
{\it Commun. Pure Appl. Anal.} {\bf 22} (2023), no. 5, 1701--1720.

\bibitem{Le_24} Le, N. Q. {\em Analysis of Monge--Amp\`ere equations.}
Graduate Studies in Mathematics, vol. 240, American Mathematical Society, Providence, RI, 2024.




\bibitem{Le_EVP} Le, N. Q. Global Lipschitz and Sobolev estimates for the Monge--Amp\`ere eigenfunctions of general bounded convex domains.
{\it Ann. Fac. Sci. Toulouse Math.} (6) {\bf 35} (2026), no. 2, 467--487.


\bibitem{Le_Var} Le, N. Q. A variational approach to degenerate Monge--Amp\`ere equations with mixed measures and monotonicity.
arXiv:2603.19114,
preprint.

\bibitem{LeSavin} Le, N. Q.; Savin, O. Schauder estimates for degenerate Monge--Amp\`ere equations and smoothness of the eigenfunctions.
{\it Invent. Math.} {\bf 207} (2017), no. 1, 389--423.

\bibitem{LeSavinSingular} Le, N. Q.; Savin, O. Global $C^{1,\beta}$ and $W^{2,p}$ regularity for some singular Monge--Amp\`ere equations.
arXiv:2407.04586, {\it Ann. Inst. Fourier (Grenoble)}, to appear.

\bibitem{Li} Li, Y. Special Lagrangian pair of pants.
{\it Comm. Pure Appl. Math.} {\bf 78} (2025), no. 7, 1320--1356.

\bibitem{Lions} Lions, P.-L. Two remarks on Monge--Amp\`ere equations.
{\it Ann. Mat. Pura Appl.} (4) {\bf 142} (1985), no. 1, 263--275.

\bibitem{LZ} Lu, C. H.; Zeriahi, A. A new approach to the Monge--Amp\`ere eigenvalue problem. arXiv:2507.18409, preprint.

\bibitem{MR} Mikhalkin, G.; Rullg{\aa}rd, H. Amoebas of maximal area.
{\it Internat. Math. Res. Notices} (2001), no. 9, 441--451.

\bibitem{Mohammed} Mohammed, A. Existence and estimates of solutions to a singular Dirichlet problem for the Monge--Amp\`ere equation.
{\it J. Math. Anal. Appl.} {\bf 340} (2008), no. 2, 1226--1234.

\bibitem{SavinZhang} Savin, O.; Zhang, Q. Boundary regularity for Monge--Amp\`ere equations with unbounded right hand side.
{\it Ann. Sc. Norm. Super. Pisa Cl. Sci.} (5) {\bf 20} (2020), no. 4, 1581--1619.

\bibitem{TongYau} Tong, F.; Yau, S.-T. Generalized Monge--Amp\`ere functionals and related variational problems.
arXiv:2306.01636, {\it Amer. J. Math.}, to appear.

\bibitem{Tso} Tso, K. On a real Monge--Amp\`ere functional.
{\it Invent. Math.} {\bf 101} (1990), no. 2, 425--448.

\bibitem{Zhou_26} Zhou, Y. Uniqueness for the degenerate Monge--Amp\`ere equation on arbitrary bounded convex domains.
arXiv:2608.22519, preprint.



\end{thebibliography}
\end{document}